\documentclass[11pt]{amsart}
\usepackage{epsfig}
\usepackage{indentfirst, latexsym, amssymb, enumerate, amsmath, graphicx}
\usepackage{colortbl}
\usepackage{subcaption}
\usepackage{mathtools}
\usepackage{mathrsfs}
\usepackage{amsfonts}
\usepackage{pgfplots}

\numberwithin{equation}{section}

\allowdisplaybreaks

\newtheorem{Theorem}{Theorem}[section]
\newtheorem{Proposition}{Proposition}[section]
\newtheorem{Lemma}{Lemma}[section]
\newtheorem{Remark}{Remark}[section]
\newtheorem{Corollary}{Corollary}[section]
\newtheorem{Definition}{Definition}[section]

\def\cosh{\mbox{{\rm cosh}}}
\def\exp{\mbox{{\rm exp}}}

\def\q{{\bf q}}
\def\x{{\bf x}}
\def\p{{\bf p}}
\def\u{{\bf u}}

\def\R{\mathbb{R}}

\def\eps{\epsilon}

\title{Global weak solutions to the relativistic BGK equation}
\author[Juan Calvo]{Juan Calvo}
\address[Juan Calvo]{\newline Departamento de Matem\'{a}tica Aplicada and Research Unit ``Modeling Nature''
(MNat), Universidad de Granada, Granada, 18071, Spain} \email{juancalvo@ugr.es}

\author[Pierre-Emmanuel Jabin]{Pierre-Emmanuel Jabin}
\address[Pierre-Emmanuel Jabin]{\newline  CSCAMM and Dept. of Mathematics,  and Research Unit ``Modeling Nature'',
                                University of Maryland,
                                College Park, MD 20742-3289, USA} \email{pjabin@umd.edu}

\author[Juan Soler]{Juan Soler}
\address[Juan Soler]{\newline Departamento de Matem\'{a}tica Aplicada and  Research Unit ``Modeling Nature''
(MNat), Universidad de Granada, Granada, 18071, Spain} \email{jsoler@ugr.es}

\begin{document}

\subjclass[2010]{35Q75, 35Q20, 76P05, 82C40} 
\keywords{BGK, Relativistic evolution equation, Boltzmann, Kinetic theory, H-theorem}

\thanks{\textbf{Acknowledgment.}  This work has been partially supported by the MINECO-Feder (Spain) research grant number MTM2014-53406-R, the Junta de Andaluc\'ia (Spain) Project FQM 954, Universidad de Granada (``Plan propio de investigaci\'on, programa 9'') through FEDER funds. Part of this work was done while J. Calvo and J . Soler were visiting CSCAMM. P.E. Jabin is partially supported by NSF Grant 1312142 and by NSF Grant RNMS (Ki-Net) 1107444.}

\begin{abstract}
In this paper the global existence of weak solutions to the relativistic BGK model for the relativistic Boltzmann equation is analyzed. The proof relies on the strong compactness of the density, velocity and temperature under minimal assumptions on the control of some moments of the initial condition together with the initial entropy.
\end{abstract}

 \maketitle
%%%%%%%%%%%%%%%%
\section{Introduction}
This paper deals with the study of weak solution of the relativistic BGK model, under minimal hypothesis of boundedness of some moments and of the entropy associated with the initial data, which allows to give a meaning to the non-linear term by means of averaging lemmas.

Relativistic gases are composed of molecules moving at speeds comparable to the speed of light. Those gases feature prominently in star dynamics, galaxy formation, free-electron lasers, high energy particle beams, controlled thermonuclear fusion and other topics... The standard tool to describe gas dynamics, be it classical, quantum or relativistic, is kinetic theory. 

Historically, classical kinetic theory was developed earlier; the central object of the theory is the so-called distribution function, a density over phase space describing the number of Newtonian gas particles in an infinitesimal volume element about a given point in phase space. Arguably the whole subject started with the early works of Maxwell and Boltzmann, who posed an evolution equation for the distribution function of a rarefied gas (the celebrated Boltzmann equation, where dynamics are mainly driven by binary collisions) together with the H-theorem about the relaxation to equilibrium. It also follows from the theory that a gas locally close to equilibrium can be well described as a fluid. This fruitful connection between classical kinetic theory and fluid dynamics has been established by various developments on the theory of hydrodynamic limits, see e.g. \cite{Cercignani, Golse2005,Sone}.
%CercignaniKremer,Crlibro
%}.
 Furthermore, this connection has inspired a whole chapter in computational fluid dynamics; the so-called lattice Boltzmann schemes  
 simulate a given fluid taking advantage of the fact that the fluid can be described as some limit regime of a gas and therefore using some discrete realization of Boltzmann's equation. However, more often than not the computational implementation of Boltzmann's gas dynamics constitutes a delicate problem. One way out of it is given by the so-called model equations: kinetic equations that are conceptually simpler but nevertheless have some properties in common with Boltzmann's equation, particularly some form of the H-theorem and their behavior on the hydrodynamical regime. 

Perhaps the most popular model equation is the one introduced by Bhatnagar, Gross, Krook \cite{BGK} and Welander \cite{Welander}, the BGK model for short (for other model equations see e.g. \cite{Cercignani}). The idea is to 
%only take 
take into account just the global effect of fluid particle interactions. This is done by means of a collision operator that replaces the complicated integral describing two-body interactions in Boltzmann's equation by a relaxation operator depending only on macroscopic quantities. This operator is constructed in such a way that mass, momentum and energy conservation hold, together with an entropy dissipation property. Despite the apparent simplicity of this representation, it is able to replicate most of the basic hydrodynamics properties (see the study of hydrodynamic limits in \cite{Laure, Laure2003} -see also \cite{Abdel1999, Abdel2004, Abdel2010} and references therein-),
 which has constituted an obvious motivation for its study. From the numerical point of view, the BGK collision operator is more amenable than Boltzmann's collision integral. Therefore, the BGK model is used as the basis of a number of lattice Boltzmann schemes \cite{Pareschi,Perthame1990,Succi}...
 It has been also used for the numerical simulations of dilute gases instead of Boltzmann's equation. However, the BGK collision operator is mathematically involved due to the presence of an exponential nonlinearity instead of a quadratic interaction; the first existence result for the BGK model, although  simpler than the celebrated DiPerna--Lions theory for the Boltzmann equation \cite{DL}, was derived later \cite{Perthame}.

If gas particles are moving at speeds comparable to the speed of light, the classical description in terms of Boltzmann's equation is not accurate and we must use the tools of relativistic kinetic theory instead; reference monographs for this subject are e.g. \cite{CercignaniKremer,deGroot}.
  This branch of kinetic theory revolves around relativistic generalizations of Boltzmann's equation for the relativistic phase
distribution $f(t,\x,\q) \ge 0$ depending on time $t \in [0,\infty)$, space $\x \in \R^3$  and momentum $\q \in \R^3$. If a relativistic gas is assumed to be non-degenerate (i.e. it obeys the Maxwell--Boltzmann statistics) and its dynamics are driven by binary collisions, we can describe its temporal evolution in terms of the relativistic Boltzmann equation, 
\[
\partial_t f + \frac{\q}{q^0} \cdot \nabla_{\x} f =\frac{m c^2}{q^0} Q (f, f) ,
\]
 where $m$ denotes the mass, $c$ represents the light speed in vacuum and $q^0 := c \sqrt{(mc)^2 + |\q|^2}$.  
Here $Q(f,f)$ denotes the non-linear quadratic (binary) collision term of the Boltzmann equation, which incorporates the intrinsic properties such that the conservation laws for particle number and energy--momentum tensor hold for this model.

Global steady states of this model are the well-known J\"uttner equilibria, also known as relativistic Maxwellians, which describe the state of
a relativistic gas in equilibrium, depending on five parameters: density $n\ge 0$, inverse temperature $\beta > 0$ and velocity $\u \in \R^3$, as follows
\begin{equation} \label{JD}
 J(n,\beta,\u;\q) = \frac{n}{(mc)^3 M(\beta)} \exp \left\{- \frac{\beta}{m c^2} (\sqrt{1+|\u|^2} c \sqrt{(mc)^2+|\q|^2} - \u \cdot \q) \right\}, \end{equation} 
where 
\begin{eqnarray}
\label{eme}
 M(\beta)
 =\int_{\R^3}\exp\left\{ -\beta \sqrt{1+|\p|^2}\right\}d
\p.
\end{eqnarray}
Note that with this notation $\beta$ is
dimensionless and so is $ M(\beta)$; the equilibrium temperature is actually given by $m c^2 /(k_B \beta)$ with $k_B$ the Boltzmann constant. 

Let us mention here some results about the relativistic Boltzmann equation in the literature. In a global regime very close to a J\"uttner distribution, Dudy\'nski and Ekiel-Je$\dot{\rm z}$ewska \cite{DE2} proved that the linear relativistic Boltzmann equation admits unique solutions in $L^2$. The existence of global-in-time renormalized 
solutions, {\it \`a la} DiPerna-Lions \cite{DL}, for large data 
 were shown  by the same authors in \cite{DE4}, using the causality of the relativistic Boltzmann equation \cite{DE2,DE3}. 
In  \cite{GS1}, Glassey and Strauss proved  the global existence, uniqueness and stability in a periodic domain of smooth solutions that are initially close to a relativistic Maxwellian.  
The case of the whole space was considered in \cite{GS2}, while the extension to the relativistic-Vlasov--Maxwell--Boltzmann  and the relativistic-Vlasov--Maxwell--Landau equations were analyzed by Guo and Strain in \cite{GSt1,GSt2}.
In \cite{A}, Andr\'easson proved the $L^1$ convergence to equilibrium for large initial data that are not necessarily close to an equilibrium solution. In \cite{Strain1,Strain2}, Strain studied the soft potential relativistic Boltzmann equation, by proving global existence, uniqueness, and rapid time convergence rates for close-to-equilibrium solutions.  The study of limit models of relativistic Boltzmann equations under physically relevant regimes has also been conducted. Newtonian limits have been reported in \cite{C,GS1,HKLN,Strain}, for different regimes. The hydrodynamic limit to the relativistic Euler fluid equations has been worked out in \cite{Speck} -see also \cite{Calvo}. 

Many of the computational methods that have been developed for the relativistic Euler equations are based on macroscopic, continuum descriptions -see \cite{Marti} for a review. However, there is room for the development of numerical schemes based on model equations for the relativistic Boltzmann equation, e.g. \cite{Chen,Kunik2003a, Kunik2003b,Kunik2004,Mendoza,Mendoza_bis}. 
 It is therefore interesting to develop our mathematical understanding about relativistic generalizations of the classical BGK model. We shall adopt here a description  
based on the Marle model \cite{Marle1,Marle2}, which can be written as
\begin{equation*}
\label{relBGK}
\partial_t f+ \frac{\q}{q^0} \cdot\nabla_{\x} f= \frac{mc^2\omega}{q^0
} (J_f - f),
\end{equation*}
where $\omega$ denotes the collision frequency. This relativistic BGK model satisfies the same conservation laws that the relativistic Boltzmann equation, as we explain below. {For simplicity and for the kind of analysis that this paper proposes to develop, from now on we can consider a rescaling of the variables so that the physical constants are all equal to one. }

The existing mathematical literature covers different aspects of the relativistic BGK model, 
see \cite{BCNS,Kunik2004} and the references therein. Defining the physical parameters of the model correctly is of great importance when paramount issues such as their relationship with relativistic macroscopic models, as for example with the Euler equations, are addressed. In this sense, the relativistic BGK model is the mesoscopic key to understanding the dynamics of relativistic fluids, as we have pointed out before. The aforementioned analysis of the model's physical parameters can be found in \cite{BCNS},
a study that becomes essential for the scaling and analysis of the classical, ultra-relativistic and hydrodynamical limits. In \cite{BCNS} it is also studied the maximum entropy principles, as well as the analysis of the linearized operator and the existence of the linearized BGK relativistic model near the global J\"uttner distribution.
The global existence of the nonlinear relativistic BGK model  together with fast-in-time decay with any polynomial rate of convergence to equilibrium,  for a close-enough to equilibrium family of initial conditions, have been analyzed by Bellouquid, Nieto and Urrutia in \cite{BNU}, using parallel arguments to those in \cite{Strain1,Strain2}. {The existence of steady state solutions has been analyzed in \cite{Hwangbis} for the problem in a slab with inflow boundary conditions}.

We note that in the case of the classical BGK model, Perthame \cite{Perthame} established the global existence of weak solutions by using an approximate BGK operator that truncates the temperature. The strong convergence of the moments (mass, velocity and temperature) was derived via an averaging  lemma. 

The aim of this paper is to explore the former ideas in the relativistic context. The difficulties to extend this classical result to the relativistic case are multiple. To begin with, the density-to-momenta map is not Lipschitz continuous, as opposed to the case of the classical BGK model. Therefore, one has to find suitable regularizations such that the truncated relativistic BGK operator gives rise to a well-posed approximate problem. Then, the standard averaging results should be generalized to the relativistic case and, even proving strong convergence of the moments, we have to identify the nonlinearity in the limit, which does not depend on the moments in a direct way. All in all, the extension of these techniques to prove existence for other model equations in relativistic kinetic theory -e.g. the Anderson--Witting model \cite{Anderson,CercignaniKremer,Hwang}- seems feasible.

The paper is structured as follows: Section \ref{sec:2} is devoted to introduce the notations and basic objects that are needed to write down the BGK--Marle model and our global existence result. Existence will be shown by means of an approximating scheme. This approximating scheme is introduced and studied in Section \ref{sec:3}. Section \ref{sec:4} investigates the time evolution of the entropy functional and the a priori estimates stemming from it. This is crucial as it enables to handle the nonlinearities in Section \ref{sec:5}, where we pass to the limit in the approximating scheme to construct solutions of the original BGK--Marle model.

%%%%%%%%%%%%%%%%%%%%%%%%

\section{Preliminaries and statement of the problem}
\label{sec:2}
%%%%%%%%%%%%%
\subsection{General conventions}
From now on, generic positive constants will be denoted by $C$, their value may change from line to line. We will write $C(a,b,\ldots)$ if we are to specify that the expression of $C$ depends on the quantities $a,b$, etc. We write $\chi_A$ for the indicator function of a set $A \subset \R^d$, that is, $\chi_A(x)=1$ if $x\in A$ and zero otherwise. We will say that two real functions $f$ and $g$ are asymptotically equivalent at zero (resp. infinity) if
$$
\lim_{x\to 0} \frac{f(x)}{g(x)} = 1\quad \mbox{resp.}\quad \lim_{x\to \infty} \frac{f(x)}{g(x)} = 1
$$
and we shall denote this just by $f \sim g$; the context will make clear if we refer to equivalence at zero or infinity. 

The space-time coordinates in the four-dimensional Minkowsky's space $\mathbb{M}$ are $x^\mu$, $\mu = 0, 1, 2, 3$, with
$x^0 =t$ for the time and $x^1, x^2, x^3 $ for the position. The metric tensor $g_{\mu \nu}$ and its inverse $ g^{\mu \nu}$ are given by
$$
g_{\mu \nu}= g^{\mu \nu}=1, \ \mbox{  if} \ \mu=\nu =0, \quad -1, \ \mbox{  if} \ \mu=\nu=
1,2,3 \quad \mbox{and}  \quad 0, \ \mbox{  if} \ \nu \neq \mu.
$$
By default Greek indices will run from $0$ to $3$. With the aid of the metric tensor we can  perform the operations of \emph{raising} and \emph{lowering} indices. That is, for any four-dimensional vector $v^\mu$ (four-vector hereafter),
$$
  g_{\alpha \nu}v^\nu = v_\alpha \quad \mbox{and}\quad g^{\alpha \nu}v_\nu = v^\alpha.
  $$
Here and in the sequel we use  Einstein's summation convention,
meaning that any index that appears twice in an expression (once as a sub-index and once as a super-index),
is understood to be summed over its whole range. In the sequel we understand $v^\mu$ as a (four)-vector and $v_\mu$ as the associated covector. We will always work on Minkowsky's space, hence $v^\alpha=-v_\alpha$ if $\alpha\neq 0$ and $v^0=v_0$. This works in the same way for general tensor objects.

We will also consider vectors in the Euclidean three-dimensional space, which we will always denote by bold characters. Then we use the standard notations $|\mathbf{v}| $ and $\mathbf{v} \cdot \q $ for the euclidean norm and scalar product respectively. Furthermore, we will restrict ourselves to work with unit rest mass particles. For that aim, we shall consider
  the following subset of Minkowsky's space (recall that the tangent space at any $x^\mu \in \mathbb{M}$ is itself isometric to $\mathbb{M}$):
$$
\begin{array}{cl}
\mathbb{M}_1 & := \{q^\mu \in \R^4/ q^\mu q_\mu=-1\}
\\ 
& \hspace{0.1cm} = \{q^\mu \in \R^4/ q^\mu=(\sqrt{1+|\q|^2},\q) \, \mbox{for some}\ \q \in \R^3\}.
\end{array}
$$
This is a three-dimensional timelike sub-manifold of Minkowsky's spacetime.

In relativistic kinetic theory, distribution functions and their (evolution) equations are defined over the tangent bundle of the underlying spacetime, whose structure may depend itself on the distribution function -e.g. the case of Vlasov--Einstein's kinetic model \cite{Andreasson, Choquet-Bruhat, Rein}. However, when gravitational effects are not relevant (i.e. we are in the framework of special relativity, that is, the underlying spacetime is $\mathbb{M}$ no matter the distribution function under consideration) the tangent bundle is diffeomorphic to $\mathbb{M}\times \mathbb{M}$ -which, as a set, is just $\R^8$. It is therefore customary to regard distribution functions to be defined on classical function spaces as it is done in non-relativistic kinetic theory. For that, we first restrict to the (future-pointing) {\em mass shell}, that is $\mathbb{M}\times \mathbb{M}_1$, which is a geodesically invariant, seven-dimensional manifold of the tangent bundle. This corresponds to particles with unit rest mass that move forward in time. Let $(x^\mu,q^\mu)$ denote the coordinate frame on the tangent bundle naturally induced by the coordinates $x^\mu$ on the base space $\mathbb{M}$. Then $(q^1, q^2, q^3)$ constitutes an orthonormal frame in $\R^3$ as a subset of the tangent space. Hence we can identify the standard time ($x^0$, that we rename as $t$), space ($x^1, x^2 $ and $x^3$, which we denote collectively by $\x$) and momenta coordinates $q^1, q^2, q^3$ thanks to the fact that $\mathbb{M}_1$ is diffeomorphic to $\R^3$ under the correspondence $q^\mu \mapsto \q$ and its inverse $\q \mapsto (\sqrt{1+|\q|^2},\q)$. Thus, hereafter distribution functions are defined as functions $f(t,\x,\q)$ where $t\ge 0$ and $\x, \q \in \R^3$.

{Also, as it was pointed out in the introduction, all the physical parameters (including the speed of light $c$) are renormalized taking, for simplicity, the value one.}
We consider all the physical quantities in dimensionless form. 

%%%%%%%%%%%%%%%%%
\subsection{Matter quantities}
We introduce the relativistic phase
density $f(t,\x,\q) \ge 0$, which represents the density of particles with given spacetime coordinates $x^\mu=(t,\x)$ and momentum $\q \in \R^3$. We will
consider that all the gas particles
 have the same rest mass. Then
 the energy-momentum four-vector is defined as
$$
q^\mu = (q^0,  \q), \quad q^0 :=  \sqrt{1 + |\q|^2}\quad (\mbox{that is,}\, q^\mu \in \mathbb{M}_1).
$$
Let us now introduce several spacetime densities associated with $f(t,\x,\q)$.
\begin{Definition}
\label{set1}
Let $x^\mu$ such that $f(t,\x,\cdot)\ge 0$ is not identically zero. We define the particle-density four-vector $N^\mu(t,\x)$, the energy-momentum tensor and the entropy four-vector as follows:
\begin{itemize}

\item $\displaystyle N^\mu(t,\x)= \int_{\R^3} q^\mu f(t,\x,\q) \frac{d\q}{q^0
},$

\item $\displaystyle T^{\mu \nu}(t,\x)=  \int_{\R^3} q^\mu q^\nu f(t,\x,\q) \frac{d\q}{q^0},$

\item $\displaystyle S^\mu(t,\x)= - \int_{\R^3} q^\mu f(t,\x,\q) \ln \left(f(t,\x,\q)\right)
\frac{d\q}{q^0}$.

\end{itemize}
\end{Definition}

Next we can define several useful macroscopic quantities (thermodynamic fields). The fact that the proper volume element $d\q /q^0$ is invariant
with respect to
Lorentz transformations (i.e. isometries of the Minkowsky space) is a key physical feature of these definitions.

\begin{Definition}
\label{set2}
Let $f(t,\x,\q)\ge 0$ not identically zero. We define the following macroscopic quantities:

\begin{enumerate}

\item The proper particle density $n_f=\sqrt{N^\mu N_\mu}$,

\item The velocity four-vector $u_f^\mu$, given by $n_fu_f^\mu = N^\mu$.

\item The proper energy density $e_f = (u_f)_\mu (u_f)_\nu  T^{\mu \nu}$.

\item The proper pressure $p_f = \frac13 ((u_f)_\mu (u_f)_\nu  - g_{\mu \nu})T^{\mu \nu}$.

\item The proper entropy density $\sigma_{f}= S^\mu (u_f)_\mu$.

\end{enumerate}
\end{Definition}

\begin{Remark}
\label{varios1}
Several comments are in order:
\begin{enumerate}

\item 
\label{v:2}
Note that $u_f^\mu (u_f)_\mu = 1$ and then $u_f^\mu =(\sqrt{1+|\u_f|^2},\u_f)$, i.e. $u_f^\mu \in \mathbb{M}_1$. We note the following useful relation,
\begin{equation}
\label{rel1}
n_f \sqrt{1+|\u_f|^2}= \int_{\R^3} f(t,\x,\q) d\q.
\end{equation}
We shall abridge $u^\mu=u_f^\mu$ whenever clear from the context. We also point out that
\begin{equation}
\label{scalaruq}
u_\mu q^\mu \ge 1,
\end{equation}
which is a straightforward consequence of the Cauchy--Schwartz inequality for vectors in
$\mathbb{M}_1$. 

\item Since $N^\mu$ is timelike, we have that $N^\mu N_\mu >0$,
 \[
N^\mu\,N_\mu=\int_{\R^6} \frac{q^\mu\,(q')_\mu}{q^0\,(q')^0}\,f(t,\x,\q)\,f(t,\x,\q')\,d\q\,d\q',
\]
 and $q^\mu\,(q')_\mu\geq 1$ -this follows again from Cauchy--Schwartz's inequality. Hence $n_f$ given by Definition \ref{set2} is well defined and positive.

\item Keep in mind that $u_f^\mu$ is not defined for those $\x\in \R^3$ such that $n_f(\x) =0$ (but the product $n_f u_f^\mu$ is, being zero at those points).

\item The proper energy density can be rewritten as 
$$
e_f=\int_{\R^3} (u_\mu q^\mu)^2 f(t,\x,\q) {d\q\over
q^0
}
$$
and clearly $e_f \ge 0$. There always holds that $0<n_f<e_f$.

\item It is clear that 
$$
\int_{\R^3} f \frac{d\q}{q_0} \le n_f \le \int_{\R^3} f \, d\q.
$$

\end{enumerate}
\end{Remark}
%%%%%%%%%%%%%
We finish this subsection by discussing the behavior of the former quantities under Lorentz transformations.  Let $\Lambda$ be a Lorentz boost (i.e. a linear isometry with respect to the Minkowsky metric) in $\R_q^4$. The restriction of $\Lambda$ to $\mathbb{M}_1$
can be regarded as a map acting on $\R_{\q}^3$ as previously explained.
 Given any distribution function $f$, we can define a new distribution function $f_\Lambda$ by means of
$$
  f_\Lambda(t,\x,\q) = f(t,\x,\Lambda \q).
$$
Recall that $d\q/q_0$ is a Lorentz invariant measure \cite{Landau}. As
\begin{equation}
   v_\mu z^\mu =( \Lambda v)_\mu (\Lambda z)^\mu \quad \mbox{for any}\ v^\mu,\ z^\mu \in \R^4,
\label{scalarprod}
\end{equation}
we get the following well-known result.
\begin{Lemma}
 Given any distribution function $f$, the scalar quantities $n_{f_\Lambda}$, $e_{f_\Lambda}$, $p_{f_\Lambda}$, $\sigma_{f_\Lambda}$ and $\beta_{f_\Lambda}$ are Lorentz invariant. The vector $u_f^\mu$ transforms according to $u_{f_\Lambda}^\mu = \Lambda^{-1}u_f^\mu$.
\end{Lemma}

%%%%%%%%%%%%%%%%%
\subsection{J\"uttner equilibria}
The generalization of the classical global Max\-wellian to Special Relativity is the so-called J\"uttner equilibrium (or relativistic Max\-we\-llian). The J\"uttner distribution \eqref{JD} can be written without physical parameters as follows
\begin{equation}
\nonumber
 J(n,\beta,\u;\q) = \frac{n}{ M(\beta)} \exp \left\{- \beta
u_\mu q^\mu \right\} \end{equation} or equivalently
$$
J(n,\beta,\u;\q) = \frac{n}{ M(\beta)} \exp \left\{- \beta \left( \sqrt{1+|\u|^2}\sqrt{1 + |\q|^2} -  \u \cdot \q \right)\right\}.
$$
Since $J(n,\beta,\u;\q)$ is thought of as an equilibrium distribution, then $n$ is interpreted as its particle density, $\u $ as the spatial part of
the four-velocity $u^\mu$ (and as such $u_\mu u^\mu = 1$) and $1/\beta$ as the equilibrium temperature. Here 
 $M(\beta)$ is given by \eqref{eme}. We also have the following relation
\begin{equation}
\label{diecinueve}
M(\beta)= {4\pi \over \beta} K_2(\beta),
\end{equation}
where the modified Bessel functions $K_j$ are defined as 
$$
K_j(\beta) =  \int_0^\infty \cosh(jr)\exp\{-\beta
\cosh(r)\}dr.
$$
The following asymptotic expansions for small and large temperature values will be helpful in the sequel. 
\begin{Lemma}
The modified Bessel function $K_1,K_2$ verify
\begin{equation}
\label{karat}
\frac{K_1}{K_2}(\beta) 
\sim 1- \frac{3}{2\beta} + O \left(\frac{e^{-\beta}}{\beta^{5/2}}\right) \quad \mbox{for}\ \beta \gg 1.
\end{equation}
\begin{equation}
\label{kanear}
 K_1(\beta) \sim \frac{1}{\beta}+O(\beta \log \beta)\quad \mbox{and}\  K_2(\beta) \sim \frac{2}{\beta^2}+O(1) \quad \mbox{for}\ \beta \ll 1,
\end{equation}
\end{Lemma}
\begin{Lemma}
The partition function $M(\beta)$ verifies
\begin{equation}
\label{Mnear}
M(\beta) \sim \frac{4 \pi}{\beta} \left(\frac{2}{\beta^2}+O(1)\right) \quad \mbox{for}\ \beta \ll 1,
\end{equation}
\begin{equation}
\label{Mfar}
M(\beta) \sim \left( \frac{2 \pi}{\beta}\right)^{3/2} e^{-\beta} + O \left(\frac{e^{-\beta}}{\beta^{5/2}}\right) \quad \mbox{for}\ \beta \gg 1,
\end{equation}
\end{Lemma}
We list below several basic properties of the J\"uttner equilibrium that will be useful in the sequel.
\begin{Lemma}
The following assertions hold true:
\begin{enumerate}
\item $M(\beta_1)<M(\beta_2)$ for $\beta_1>\beta_2$. 

\item In the Lorentz rest frame the J\"uttner equilibrium reduces to
$$
J(n,\beta,0;\q)= \frac{n}{M(\beta)} \exp \{- \beta \sqrt{1 + |\q|^2} \}.
$$

\item $J\le n e^{-\beta}/M(\beta).$
\end{enumerate}
\end{Lemma}

Some moments of the J\"uttner distribution are easily computed. Namely:
\begin{Lemma}
\label{Jmoments}
Define the function $\Psi$ as
\begin{equation}
\label{defpsi}
 \Psi(\beta)= \frac{3}{\beta} + \frac{K_1(\beta)}{K_2(\beta)}.
 \end{equation}
 Then the following identities hold:
\begin{enumerate}
\item
$
e_{J}=n\Psi(\beta),
$
\item
$
\displaystyle p_{J}= \frac{n}{\beta},
$

\item
\label{id:3}
$\displaystyle
\int_{\R^3} q^\mu J \frac{d\q}{q^0} = n u^\mu,
$
%%%%%%%
\item
\label{id:7}
$\displaystyle
\int_{\R^3} J \frac{d\q}{q^0}  = e_{J} - 3 p_{J}= n \left( \Psi(\beta) - \frac{3}{\beta} \right) = n \frac{K_1(\beta)}{K_2(\beta)}.
$
\end{enumerate}
\end{Lemma}
%

%%%%%%%%%%%%%%%%%

%%%%%%%%%%%%%%%
\subsection{The BGK-Marle model}
%%%%%%%%%%%%%
We consider the BGK-Marle model in the following form:
\begin{equation}
 \partial_t f + \frac{\q}{q^0}\cdot \nabla_{\x} f = \frac{J_f - f}{q^0},
 \label{modelo}
\end{equation}
where  the J\"uttner local equilibrium $J_f$ is constructed from some macroscopic invariants of the function $f(t)$. More precisely,
\begin{equation}
\label{eq:local_eq}
J_f(t,\x,\q) = {n_f(t,\x) \over M(\beta_f(t,\x))}\exp\{-\beta_f(t,\x) (u_f(t,\x))_\mu  q^\mu\}.
\end{equation}
The function $\beta_f(t,\x)$ is defined by means of the relation
\begin{equation}
\label{betadef}
   \frac{K_1(\beta_f)}{K_2(\beta_f)} = \frac{\displaystyle \int_{\R^3}  f \frac{d\q}{q^0
}}{n_f}.
\end{equation}
It is straightforward to check that this relation defines $\beta_f$ uniquely due to the following result.
\begin{Lemma}[\cite{BCNS}]
\label{lm:BCNS}
The function $\xi\rightarrow \frac{K_1(\xi)}{K_2(\xi)}$ is strictly increasing and one-to-one from $[0,\infty)$
 to $[0,1)$. 
\end{Lemma}

So defined, the right hand side of the equation (\ref{modelo}) verifies the following cancellation/conservation properties:
\begin{equation}
\label{veintiocho}
\int_{\R^3} q^\mu J_f {d\q\over q^0}= \int_{\R^3} q^\mu f {d\q\over q^0}, 
\end{equation}
\begin{equation}
\label{veintinueve}
\int_{\R^3}  J_f \frac{d\q}{q^0
} =\int_{\R^3}  f \frac{d\q}{q^0}.
\end{equation}
Therefore, the relaxation operator is determined in such a way that the  five conservation laws for the particle number, the energy, and the momentum hold, which is to say, the solutions to (\ref{modelo}) satisfy the following equation in divergence form
\begin{equation}
\label{conservationlaws}
 \sum_\nu {\partial N^\mu \over \partial x^\mu}=0, \quad \sum_\nu {\partial T^{\mu \nu} \over \partial
x^\nu}=0\:.
\end{equation}
These conservation laws are derived from the fact that the particles interact only through elastic collisions, without other forces and/or  radiation being involved. 
  
J\"uttner equilibria associated with a given distribution function satisfy a couple of useful extremality principles, as we now state.
\begin{Proposition}
Let $0\le f\in L^1(\R_{\q}^3)$ be given and let $J_f$ be the associated J\"uttner equilibrium defined by \eqref{eq:local_eq}. Then there holds that:
 \begin{enumerate}
 \item $
\left(\sigma-\beta e\right)_{J_f} - \left(\sigma-  \beta e\right)_f \ge 0.
$

\item 
\label{optimality}
$
\displaystyle \int_{\R_{\q}^3} J_f \log J_f \frac{d\q}{q^0} \le \int_{\R_{\q}^3} f \log f \frac{d\q}{q^0}.
$
 \end{enumerate}\label{entropymin}
\end{Proposition}
\begin{proof}
The first point can be found in \cite{BCNS}. The proof of the second follows the lines of \cite{Kunik2004}:
As $x  \mapsto x \log x$ is a convex function, we get
$$
  f \log f \ge J_f \log J_f + \left. \frac{d (x \log x)}{dx} \right|_{J_f}(f - J_f),
$$
that is 
$$
 f \log f \ge J_f \log J_f + \left(1+\log \frac{n_f}{M(\beta_f)} -\beta_f (u_f)_\mu q^\mu \right)(f - J_f).
$$
We check that
$$
\int_{\R_{\q}^3} \left(1+\log \frac{n_f}{M(\beta_f)} -\beta_f (u_f)_\mu q^\mu \right) (f - J_f) \frac{d\q}{q^0} = 0
$$
thanks to \eqref{veintiocho} and \eqref{veintinueve}. The result follows.
\end{proof}

%%%%%%%%%%%
\subsection{Main result and comments on the proof}
Let us first introduce our notion of solution:
\begin{Definition}
Let $0\le f^0 \in L^1 (\R^6)$ and consider $T>0$. A function $f \in C([0,T),L^1(\R^6))$ is a weak solution of \eqref{modelo} in $[0,T)\times \R^6$ with initial datum $f^0$ if $f(t=0)=f^0$, $f(t)\ge 0$ a.e in $\R^6$ for every $t \in [0,T)$ and
$$
\int_0^T \int_{\R^6} f \partial_t \phi + f \frac{\q}{q^0}\cdot \nabla_{\x} \phi \, d\q \,d\x \,dt + \int_{\R^6} \phi(t=0) f^0\, d\q \, d\x= \int_0^T \int_{\R^6} \phi \frac{f-J_f}{q^0}\, d\q \, d\x \, dt
 $$
holds for every $\phi \in C_c^1([0,T)\times \R^6)$.
\end{Definition}

We can now state the main result of this document.
\begin{Theorem}
Let $f^0\ge 0$ a.e. $\R^6$ be such that
$$
\int_{\R^6} (1+q^0+|\x|+\log f^0) f^0\, d\q \, d\x <\infty.
$$
Then, given $T>0$ there exists a weak solution $f:[0,T)\times \R^6 \to \R_+$ of \eqref{modelo}  with initial datum $f^0$. {Furthermore, this weak solution satisfies an H-theorem,} in the sense that
$$
t \mapsto \int_{\R^6} f(t,  \x , \q ) \log f(t,  \x , \q)\, d\q \, d\x \quad \mbox{is a nonincreasing map for}\ t \in [0,T).
$$
\end{Theorem}

The main idea of the proof is to build an appropriate functional environment to be able to apply fixed-point theorems. The difficulty comes from controlling the relaxation term in all its variables, especially those related to temperature and speed. To solve this problem it is necessary to truncate the relaxation term adequately, and work with weighted $L ^1$ spaces with respect to the proper volume element $d\q /q^0$ -which is invariant with respect to Lorentz transformations. The truncated thermodynamical fields satisfy density, temperature and velocity stability estimates that ensure the well-posedness of the truncated system. The corresponding approximate system then depends on the truncating parameters, and the first objective is to estimate these solutions and control some of their moments in term of these parameters. This control will allow to adapt the orders of magnitude of the parameters jointly, so that in the limit the moments associated with the approximate solutions can be estimated appropriately. To avoid concentrations in the limit we must also control the entropy in the approximate system independently of the truncation parameters. This analysis also allows us to obtain an H-theorem for the evolution of the distribution function.
%%%%%%%%%%%%%%%%%%%%%%%%
\section{Set-up for an approximating scheme}
\label{sec:3}
The aim of this section is to study the  the following approximated problem:
\begin{equation}
\label{apro}
 \partial_t f + \frac{\q}{q^0}\cdot \nabla_{\x} f = \frac{\tilde{J}[f] - f}{q^0},
 \end{equation}
which will lead to an iterative scheme to build the solutions of \eqref{modelo}. The definition of the truncated relaxation operator $\tilde{J}[f] $ depends on three parameters $R, L,\beta_{sup}>1$ and a cutoff function $\varphi$. Let
$$
\tilde{J}[f] : = \varphi(\q) \frac{n_f}{\tilde{M}(\tilde{\beta}_f)} e^{-\tilde{\beta}_f (\tilde{u}_f)_\mu q^\mu},
$$
where:
\begin{itemize}

\item The cutoff $0\le \varphi \le 1$ is a smooth function such that $\varphi(\q)=1$ if $|\q|<R$ and $\varphi(\q)=0$ if $|\q|>2R$. More specifically, we pick $0\le \varphi_0 \le 1$ a smooth, radially symmetric, decreasing outwards function such that $\varphi_0(\q)=1$ if $|\q|<1$ and $\varphi_0(\q)=0$ if $|\q|>2$ and we let $\varphi(\q):=\varphi_0(\q/R)$ (note that we omit the $R$-dependence in the notation).

\item We define $\tilde{M}(\beta):= \int_{\R_{\q}^3} \varphi(\q) e^{-{\beta}  q^0}\, d\q$.
\item Let $\beta_{inf}:=1/\beta_{sup}$ and $\tilde{\beta}_f = \left\{ 
\begin{array}{ll}
\beta_{sup} & \mbox{if}\  \beta_f > \beta_{sup},
\\
\beta_f & \mbox{if}\ \beta_{inf}\le \beta_f \le \beta_{sup},
\\
\beta_{inf} & \mbox{if}\  \beta_f < \beta_{inf}.
\end{array}
\right.
$

\item The truncated four-velocity is defined through $\tilde{\u}_f = \left\{ 
\begin{array}{ll}
\u & \mbox{if}\  |\u|\le L,
\\
L \frac{\u}{|\u|} & \mbox{if}\  |\u|> L,
\end{array}
\right.
$ and $\tilde{u}_0= \sqrt{1+|\tilde{\u}|^2}$. 
\end{itemize}
\begin{Remark}
The following properties will be useful in the sequel:
\begin{enumerate}
\item $\tilde{M}(\tilde{\beta_f}_1)<\tilde{M}(\tilde{\beta_f}_2)$ for $\tilde{\beta_f}_1>\tilde{\beta_f}_2$, 
\item $\tilde{M}(\beta) \le M(\beta)$ for every $\beta \in (0,\infty)$,
\item $\tilde u^\mu$ so defined verifies $\tilde u^\mu \tilde u_\mu = 1$.
\end{enumerate}
\end{Remark}
The initial datum $f^0$ is regularized by truncation to ensure that the $\q$-support is contained in $\{|\q|\le 2R\}$. 

In what follows we will impose some constraints on $R, L$ and $\beta_{sup}$ in order to have a single regularizing parameter in our approximating scheme. We shall define
\begin{equation}
\label{eq:fix}
 R:= \beta_{sup}^2\, ,\quad L:= \beta_{sup}.
\end{equation}
Although we impose \eqref{eq:fix} to hold during the rest of the document, we will keep the notations $R, L$ at those parts where we find it informative. 

%%%%%%%%%%%%%
The main result of this section is the following existence result for the approximating scheme.
\begin{Theorem}
\label{aprox}
If $f^0$ is supported in $\R^3 \times \{|\q|\le 2R\}$, then  there exists a unique solution $f \in C([0,T),L^1)$ to \eqref{apro}.
\end{Theorem}
\begin{proof}
Given $f \in L^\infty([0,T); L^1(\R^3 \times \{|\q|\le 2R\}))$ we define $T[f]$ as the solution of 
\[
 \partial_t T[f] + \frac{\q}{q^0}\cdot \nabla_{\x} T[f] = \frac{\tilde{J}[f] - T[f]}{q^0}
\]
with initial datum $f^0$. The previous system can be solved using classical arguments of kinetic theory that involve analyzing the associated characteristic dynamic system, whose transport field $\frac{\q}{q^0}$ is regular in this case. Therefore, the characteristic system is bounded, 
 from which the necessary bound in $L^1$ is derived and then the {\it a priori} estimate in time in $W^{1,1}$, which provides the continuity in time with values in $L^1$ in momentum and space.

The main tool in the proof is the following lemma
\begin{Lemma}
\label{lm:stab}
  Let $f_1,f_2\in L^1(\{|\q|\le 2R\})^+$ with associated thermodynamical fields $n_i = n_{f_i}$, $\u_i=\u_{f_i}$, $\beta_i= \beta_{f_i}$, $i=1,2$. The truncated relaxation operator satisfies the following stability estimate
  \[
  |\tilde J[f_1]-\tilde J[f_2]|\leq C_1(\q)\, \frac{\varphi(\q)}{\tilde{M}({\beta}_{sup})}\,e^{-{\beta}_{inf} |\q|/3L}\,\int_{|\q|\le 2R} |f_1-f_2|\,d\q,
  \]
  where
  \begin{eqnarray*}
C_1(\q)&\hspace{0.1cm}=&C_1(\q,R,L,\beta_{sup},\varphi)\\&:=&2\,\sqrt{1+4R^2}+\left(1+2\,\sqrt{1+4R^2}\,\sqrt{1+L^2}\right)\,\left( 2 R C_2 \frac{\tilde M(\beta_{inf})}{\tilde{M}({\beta}_{sup})}+ \beta_{sup}\,q^0\right)
\end{eqnarray*}
\label{LipschitztildeJ}
and $C_2$ is defined in Lemma \ref{lipmoments} below.
  \end{Lemma}
This lemma requires a more thorough analysis of the a priori estimates which is conducted in the rest of the section with the proof of the lemma at the end of subsection~\ref{Lipschitzbounds}.

Thanks to Lemma \ref{LipschitztildeJ}, we have the estimate
\begin{eqnarray*}
\int_{\R^6} |T[f]-T[g]|(t)\, d\x d\q &\le& \int_0^t \int_{\R^6} \left|\tilde{J}[f]-\tilde{J}[g] \right|(\tau)\, d\x \frac{d\q}{q^0} d\tau\\
&\le& C_3 \int_0^t \int_{\R^6} |f-g|(\tau)\, d\x \,d\q \, d\tau
\end{eqnarray*}
with
\[
C_3=C_3(\q,R,L,\beta_{sup},\varphi):= \int_{\R_{\q}^3}  C_1(\q)\, \frac{\varphi(\q)}{\tilde{M}({\beta}_{sup})}\,e^{-{\beta}_{inf} |\q|/3L} \frac{d\q}{q^0} <\infty,
\]
where $C_1$ is defined in Lemma \ref{lm:stab} above. Then, we are entitled to invoke Picard' fixed point theorem together with a prolongation argument to deduce the well-posedness of the problem.
\end{proof}
The aim of the rest of the section is to build up the cascade of estimates that will ultimately lead to the Lipschitz properties of $T[f]$ as stated in Lemma \ref{LipschitztildeJ}.
%%%%%%%%%%%%%%%%%%%%%%%%

\subsection{A priori estimates}

\begin{Lemma}
\label{moment_estimates}
Let $f^0(\x,\q)$ supported in $\R^3 \times \{|\q|\le 2R\}$ be given. Let $f(t,\x,\q)$ be a solution to \eqref{apro} with $f^0$ as initial datum. Then,
\begin{enumerate}

\item If $ f^0\ge 0$, then $f(t)\ge 0, \quad \forall \ t\ge 0$.

\item If $\displaystyle  \int_{\R_{\x}^3} \int_{\R_{\q}^3} (1+|\x|+q^0) f^0\ d\x \, d\q <\infty$, then $\displaystyle  \int_{\R_{\x}^3} \int_{\R_{\q}^3} (1+|\x|+q^0) f(t)\, d\x \, d\q$ is bounded on bounded time intervals.

\end{enumerate}
\end{Lemma}

\begin{proof}
Non-negativity follows by writing the solution in terms of characteristics, so that clearly
$$
\frac{d}{dt} [f(t,\x+\q t/q^0,\q)] \ge -\frac{1}{q_0}\, f (t,\x+\q t/q^0,\q)
$$
and hence
$$
f(t,\x,\q) \ge e^{-t/q_0} f^0(\x-\q t/q^0,\q)\ge 0,\quad \forall t\ge 0.
$$

Let us develop now the moment estimates. We start by integrating in \eqref{apro}:
$$
\frac{d}{dt} \int_{\R^6} f \, d\q\, d\x \le \int_{\R_{\x}^3} \ n \frac{K_1}{K_2}(\tilde \beta) \, d\x \le \int_{\R_{\x}^3} \sqrt{1+|\u|^2} n \, d\x =\int_{\R^6} f\,d\q\,d\x,
$$
by using \eqref{rel1}. We hence easily find that
$$
\int_{\R^6} f \, d\q\, d\x \le e^t \int_{\R^6} f^0 \, d\q\, d\x.
$$
Next we multiply \eqref{apro} by $|\x|^k, k >0$ and integrate to find
$$
\frac{d}{dt} \int_{\R^6} |\x|^k f \, d\q\, d\x - \int_{\R^6}\frac{\q}{q^0}\cdot \nabla (|\x|^k) f \, d\q \, d\x \le k \int_{\R_{\x}^3} |\x|^{k-1} n \, d\x \le k \int_{\R^6} |\x|^{k-1} f \, d\q\, d\x.
$$
If we choose $k=1$ we get 
$$
\frac{d}{dt} \int_{\R^6} |\x| f \, d\q d\x  \le \int_{\R^6} f \, d\q \, d\x.
$$
Using the previous point,
$$
\frac{d}{dt} \int_{\R^6} (1+|\x|) f \, d\q\, dx  \le 2 \int_{\R^6} (1+|\x|) f \, d\q \, d\x
$$
and hence
$$
\int_{\R^6} (1+|\x|) f \, d\q\, d\x \le e^{2t} \int_{\R^6} (1+|\x|) f^0\, d\q \, d\x.
$$
Low momenta in $q^\mu$ can be controlled likewise. Multiplying \eqref{apro} by $q^0$ and integrating we get 
$$
\frac{d}{dt} \int_{\R^6} q^0  \, d\q\, d\x \le \int_{\R_{\x}^3} \tilde J[f ] \, d\x = \int_{\R^6} f \, d\q\, d\x \le \int_{\R^6} q^0 f \, d\q \, d\x,
$$
which implies
$$
\int_{\R^6} q^0 f  \, d\q \, d\x \le e^t \int_{\R^6} q^0 f^0 \, d\q \, d\x.
$$
\end{proof}
We now prove some auxiliary estimates that will help us to assess the convergence of the approximating scheme.
\begin{Lemma}
\label{l:5}
Let $L\ge 1$. Then we have that
$$
 e^{-{\beta} \tilde{u}_\mu q^\mu}\le e^{-{\beta}|\q|/(3L) }, \quad \forall \q \in \R^3.
$$
\end{Lemma}
\begin{proof}
Just follow the chain of inequalities:
\begin{equation}
\nonumber
\begin{array}{rl}
\displaystyle
\tilde{u}_\mu q^\mu & \displaystyle
= \sqrt{1+|\q|^2} \sqrt{1+|\tilde{\bf u}|^2} - \tilde{\bf u} \cdot \q
 \displaystyle
\ge \sqrt{1+|\q|^2} \sqrt{1+|\tilde{\bf u}|^2} - |\q| |\tilde{\bf u}|
\\ \\
& \displaystyle
\ge |\q| \sqrt{1+|\tilde{\bf u}|^2} - |\q| |\tilde{\bf u}|
 \displaystyle
\ge |\q| (\sqrt{1+L^2} - L)
 \displaystyle
\ge |\q|/(3 L).
\end{array}
\end{equation}
We used that $x\mapsto \sqrt{1+x^2}-x$ is decreasing for $x>0$ to replace $|\tilde{\bf u}|$ by $L$ in the last line. Last step follows from
$$
\sqrt{1+L^2} - L = \frac{1}{\sqrt{1+L^2} + L} \ge \frac{1}{3L},
$$
which holds for $L\ge 1$.
\end{proof}
To proceed further we introduce some shorthand notations for various residuals that will appear recurrently in the sequel
\begin{Definition}
Let us consider the following positive quantities:
\begin{eqnarray} \label{phi} \qquad
\Phi(R)&=&\Phi(R;L,\beta):=\int_{|\q|\ge R} e^{-\beta |\q|/(3 L)} d\q = \frac{108 \pi L^3}{\beta^3} e^{-\frac{\beta R}{3 L}} \left(\frac{\beta^2 R^2}{9 L^2} + \frac{2 \beta R}{3 L} + 2 \right),
\\
\label{lambda}
\Lambda(R)&=&\Lambda(R;\beta)
:= \int_{|\q|\ge R} e^{-\beta |\q|} d\q= \frac{4 \pi}{\beta^3} e^{-\beta R} (\beta^2 R^2 + 2\beta R +2).
\end{eqnarray}
\end{Definition}
Note that as a consequence of Lemma \ref{l:5} we have
$$
\int_{|\q|\ge R} e^{-{\beta} \tilde{u}_\mu q^\mu} d\q \le \Phi(R;L,\beta),
$$
when $L\ge 1$. We exploit this a bit further to control the approximation of $M(\beta)$ by $\tilde{M}(\beta)$.
\begin{Lemma}
\label{lcomp}
We have
$$
|\tilde M(\beta)-M(\beta)|\le \Lambda(2R).
$$
 As a consequence, there exists some $\bar \beta_1>0$ with the  property: for every pair of values $\beta \ge \bar \beta_1$ and $R>1$ such that
\begin{equation}
\label{eq:suffcond}
\frac{2^{5/2}}{\sqrt{\pi}\beta^{3/2}}(4 \beta^2 R^2+4 \beta R +2) e^{-2\beta R} \le e^{-\beta}
\end{equation}
the following inequality
\begin{equation}
\label{ratio_est}
\frac{1}{\tilde M (\beta)} 
\le \frac{2}{M(\beta)}
\end{equation}
holds.
\end{Lemma}
\begin{proof}
Clearly
$$
\tilde M(\beta)=M(\beta) + \int_{\R_{\q}^3} e^{-\beta q^0} (\varphi(\q)-1)\, d\q   
$$
and then
$$
|\tilde M(\beta)-M(\beta)|\le \int_{|\q|\ge 2R} e^{-\beta q^0} \, d\q \le \int_{|\q|\ge 2R} e^{-\beta |\q|} \, d\q=\Lambda(2R).
$$
To derive \eqref{ratio_est} we start noting that $\tilde M(\beta)\le M(\beta)$, which ensures that $\tilde M(\beta)\ge M(\beta)-\Lambda (2R)$. If we were able to find a set of values $\beta, R$ for which we had $2\Lambda(2R) \le M(\beta)$ we would be done. Let us provide a sufficient condition for that. Starting from \eqref{Mfar}, we determine some $\bar \beta_1$ large enough so that 
$$
M(\beta) \ge \frac{1}{2} \left( \frac{2 \pi}{\beta}\right)^{3/2} e^{-\beta},\quad \mbox{for every}\ \beta\ge \bar \beta_1,
$$
 and then the sufficient condition given in \eqref{eq:suffcond} follows.
\end{proof}
\begin{Remark}
For future usage we note that under the constraint $R=\beta^2$, see \eqref{eq:fix}, the condition \eqref{eq:suffcond} is already satisfied by any $\beta \ge 2$. There is no loss of generality in assuming that $\bar \beta_1 >2$.  
\end{Remark}
%%%%%%%%%%%%%%%%%%%%%%%%%%%%%%%%%%%%%
\subsection{Lipschitz bounds on the truncated relaxation operator\label{Lipschitzbounds}}
%%%%%%%%%%%%%%%%%%%%%%%%%%%%%%%%%%%%%%
The main step to obtain the existence of solutions to our approximated equation, is to derive Lipschitz bound on the relaxation term seen as an operator on $f$. The key point for so doing is the observation that when $f$ is compactly supported in $\q$ then $n_f$ can be bounded from below.

\begin{Lemma}\label{lm:}
Let $f(\q)\ge 0$ such that $f=0$ for $|\q|\ge 2R$. Then, we have
$$
n_f \ge \frac{1}{\sqrt{1+4R^2}}\int_{\R_{\q}^3} f\, d\q.
$$
\end{Lemma}
\begin{proof}
 Using a symmetry argument we can write
$$
n_f^2 = \int_{\R^3}\int_{\R^3} f(\q) f({\bf q'}) \left(1-\frac{\q \cdot {\bf q'}}{q^0 (q^0)'} \right)\, d\q \, d{\bf q'}.
$$
Then, being $R$ fixed we can show that 
$$
\int_{\R^3}\int_{\R^3} f(\q) f({\bf q'}) \left(1-\frac{\q \cdot {\bf q'}}{q^0 (q^0)'} \right)\, d\q \, d{\bf q'} \ge \frac{1}{1+4R^2}\left( \int_{\R_{\q}^3} f\, d\q \right)^2.
$$
\end{proof}

The former result enables us to obtain Lipschitz bounds on the moments as per
\begin{Lemma}
Let $f_1,f_2\in L^1(\{\q \in \R^3/|\q|\le 2R\})^+$ with associated thermodynamical fields $n_i = n_{f_i}$, $\u_i=\u_{f_i}$, $\beta_i= \beta_{f_i}$, $i=1,2$. The truncated thermodynamical fields defined after (\ref{apro}) satisfy the following stability estimates
\begin{enumerate}
\item $\displaystyle |n_1-n_2| \le 2\sqrt{1+4R^2} \int_{|\q|\le 2R} |f_1-f_2| \ d\q$.

\item $\displaystyle |\tilde \u_1-\tilde \u_2| \le  \frac{1+2 \sqrt{1+4R^2} \sqrt{1+L^2}}{\max_{i=1,2} n_i}\int_{|\q|\le 2R} |f_1-f_2| \ d\q$.

\item $\displaystyle |\tilde \beta_1-\tilde \beta_2| \le  C_2(\beta_{sup}) \frac{1+2 \sqrt{1+4R^2} \sqrt{1+L^2}}{\max_{i=1,2} n_i} \int_{|\q|\le 2R} |f_1-f_2| \ d\q$,

 with $C_2(\beta_{sup}):=\sup_{[\beta_{inf},\beta_{sup}]} \left[\left(K_1/K_2 \right)^{-1}\right]' <+\infty$.
\end{enumerate}
\label{lipmoments}
\end{Lemma}
\begin{proof}
There is no loss of generality in assuming that $n_1,n_2 >0$, as otherwise the inequalities are essentially trivial. 

\medskip
\noindent
{\em Estimate on the proper densities:} This is obtained using Lemma \ref{lm:} and  
(\ref{rel1}) 
 in turn: 
\[
\begin{split}
& |n_1-n_2|= \frac{|n_1^2-n_2^2|}{n_1+n_2} 
\nonumber \\
&\quad\le  \frac{\sqrt{1+4R^2}}{\displaystyle \int_{|\q|\le 2R} (f_1+f_2 )\ d\q} \left| \left(\int_{|\q|\le 2R} f_1 \ d\q \right)^2 -n_1^2 |\u_1|^2- \left(\int_{|\q|\le 2R} f_2 \ d\q \right)^2 + n_2^2 |\u_2|^2\right|.
\nonumber\\
\end{split}
\]
Hence
\begin{eqnarray}
|n_1-n_2| \hspace{-0,3cm}&= \hspace{-0,2cm}& \frac{\sqrt{1+4R^2}}{\displaystyle \int_{|\q|\le 2R} (f_1+f_2) \ d\q} \Bigg| \int_{|\q|\le 2R} (f_1+f_2) \ d\q   \cdot \int_{|\q|\le 2R} (f_1-f_2) \ d\q
\nonumber \\
&&\hspace{7cm} -(n_1\u_1+n_2 \u_2)|n_1\u_1 -n_2 \u_2|\Bigg|
\nonumber \\
&\le & \sqrt{1+4R^2}  \int_{|\q|\le 2R} |f_1-f_2| \ d\q 
\nonumber \\
&&\hspace{1cm}  + \frac{\sqrt{1+4R^2}}{\displaystyle \int_{|\q|\le 2R} (f_1+f_2)\ d\q} \int_{|\q|\le 2R} q^\mu (f_1+f_2) \frac{d\q}{q^0}\cdot \int_{|\q|\le 2R} q^\mu |f_1-f_2| \frac{d\q}{q^0}.
\end{eqnarray}

\noindent
{\em Estimate on the velocities:}
To estimate the difference of two velocity vectors we consider first the case $|\u_1|,|\u_2|\le L$. With no loss of generality, we may assume that $n_2\geq n_1$; then, using again (\ref{rel1}), we have 
\begin{eqnarray*}
|\tilde \u_1 -\tilde \u_2|&=&|\u_1 -\u_2| = \frac{1}{n_1n_2}\left|n_2 \int_{|\q|\le 2R} \q f_1 \frac{d\q}{q^0} - n_1 \int_{|\q|\le 2R} \q f_2 \frac{d\q}{q^0} \right| 
 \\
&= &\frac{1}{n_1n_2} \left|(n_2-n_1) \int_{|\q|\le 2R} \q f_1 \frac{d\q}{q^0} + n_1 \int_{|\q|\le 2R} \q (f_1-f_2) \frac{d\q}{q^0} \right| 
\\
&\le & \frac{\sqrt{1+L^2}}{n_2} |n_1-n_2|+ \frac{1}{n_2}  \int_{|\q|\le 2R} |f_1-f_2| \, d\q,
\end{eqnarray*}
and we conclude thanks to the former estimate for proper densities. The case $|\u_1|,|\u_2|\ge L$ can be reduced to the previous one as follows:
\begin{eqnarray*}
|\tilde \u_1 -\tilde \u_2|&=& L \left|\frac{\u_1}{|\u_1|} - \frac{\u_2}{|\u_2|}\right| = \frac{L}{|\u_1| |\u_2|}\big|\u_1|\u_2|-\u_2 |\u_1| \big|
\\
&=& \frac{L}{|\u_1| |\u_2|} \big||\u_2|(\u_1-\u_2)+\u_2 (|\u_2|-|\u_1|) \big| 
\\
&\le& \frac{2L|\u_1-\u_2|}{|\u_1|}\le 2 |\u_1-\u_2|.
\end{eqnarray*}
In that case that $|\u_1| \le L$ and $|\u_2|\ge L$, we have that
$$
|\tilde \u_1 -\tilde \u_2|= \frac{1}{|\u_2|}\big||\u_2|\u_1 - L\u_2\big| = \frac{1}{|\u_2|}\big|L(\u_1-\u_2) + \u_1(|\u_2|-L)\big|
$$
and we conclude by noting that $0\leq \left| |\u_2|-L \right|\le |\u_2|-|\u_1|$, which implies again that $|\tilde \u_1 -\tilde \u_2|\le 2 |\u_1-\u_2|$.

\noindent
{\em Estimate on the inverse temperature:}
We start with the case $\beta_{inf} \le \beta_1,\beta_2 \le \beta_{sup}$. We write
\begin{eqnarray*}
|\beta_1-\beta_2|& =& \left| \left(\frac{K_1}{K_2} \right)^{-1}\left(\frac{\int f_1 \frac{d\q}{q^0}}{n_1} \right) - \left(\frac{K_1}{K_2} \right)^{-1}\left(\frac{\int f_2 \frac{d\q}{q^0}}{n_2} \right)\right|
\\
&\le & C_2 \left|\frac{\int f_1 \frac{d\q}{q^0}}{n_1}-\frac{\int f_2 \frac{d\q}{q^0}}{n_2} \right|=\frac{C_2}{n_1 n_2}  \left| n_2 \int_{|\q|\le 2R} f_1 \frac{d\q}{q^0} - n_1\int_{|\q|\le 2R} f_2 \frac{d\q}{q^0} \right| ,
\end{eqnarray*}
with $C_2=C_2(\beta_{sup}):=\sup_{[\beta_{inf},\beta_{sup}]} \left[\left(K_1/K_2 \right)^{-1}\right]' $. Then we conclude by writing
$$
\left| n_2 \int_{|\q|\le 2R} f_1 \frac{d\q}{q^0} - n_1\int_{|\q|\le 2R} f_2 \frac{d\q}{q^0} \right| \le |n_1-n_2| \int_{|\q|\le 2R} f_1 \frac{d\q}{q^0} + n_1 \int_{|\q|\le 2R} |f_1-f_2| \frac{d\q}{q^0}.
$$
Next we consider the case with $\beta_{inf} \le \beta_1 \le \beta_{sup}$ and $\beta_2 \ge \beta_{sup}$. Note that
\begin{eqnarray*}
|\tilde \beta_1-\tilde \beta_2| \hspace{-0,2cm}&=\hspace{-0,2cm}& |\beta_1- \beta_{sup}| =\left| \left(\frac{K_1}{K_2} \right)^{-1}\left(\frac{\int f_1 \frac{d\q}{q^0}}{n_1} \right) - \left(\frac{K_1}{K_2} \right)^{-1}\left(\frac{K_1}{K_2}(\beta_{sup}) \right)\right|
\\
\hspace{-0,2cm}&\le\hspace{-0,2cm}& C_2 \left|\frac{\int f_1 \frac{d\q}{q^0}}{n_1} - \frac{K_1}{K_2}(\beta_{sup})\right|
= \frac{K_1}{K_2}(\beta_{sup}) - \frac{\int f_1 \frac{d\q}{q^0}}{n_1}
\le \frac{\int f_2 \frac{d\q}{q^0}}{n_2} - \frac{\int f_1 \frac{d\q}{q^0}}{n_1} ,
\end{eqnarray*}
where we used the monotonicity of $K_1/K_2$. Then we conclude as in the former case. The remaining cases can be dealt with in a similar way.
\end{proof}
From this, we may now finish the analysis in this section with the proof of Lemma \ref{lm:stab} which was used in the proof of our main theorem. 

\begin{proof}[Proof of Lemma \ref{lm:stab}.]
  We recall that
  \[
\tilde{J}[f] : = \varphi(\q) \frac{n_f}{\tilde{M}(\tilde{\beta}_f)} e^{-\tilde{\beta}_f (\tilde{u}_f)_\mu q^\mu},
\]
and that $\tilde u_\mu\,q^\mu\geq |q|/(3L)$. Observe that
\[
|\tilde M'(\beta)|=\int_{\R_{\q}^3} \varphi(\q)\,e^{-\beta\,q^0}\,q^0\,d\q\leq 2\,R\,\tilde M(\beta),
\]
so that
\[\begin{split}
&|\tilde{J}[f_1]-\tilde{J}[f_2]|\leq \frac{\varphi(\q)}{\tilde{M}({\beta}_{sup})}\,e^{-{\beta}_{inf} |\q|/3L}\,\Big[ |n_1-n_2|\\
  &\qquad+\,2\,R\,(n_1+n_2)\,\frac{\tilde M(\beta_{inf})}{\tilde M(\beta_{sup})}\, |\tilde\beta_1-\tilde\beta_2|+(n_1+n_2)\,\beta_{sup}\,q^0\,|\u_1-\u_2| \Big].\end{split}
\]
We now use Lemma \ref{lipmoments} to conclude.
\end{proof}

%%%%%%%%%%%%%%%%%%%
\section{Entropy estimates}
\label{sec:4}
The main aim of this section is to prove the following statement:
\begin{Proposition}
\label{entropy_estimate}
Let $f$ be a solution to (\ref{apro}) with initial datum $f^0\ge 0$ such that
$$
\int_{\R^6} (1+q^0+|\x|+\log f^0) f^0\, d\q d\x <+\infty.
$$
Then, there exists some $\bar \beta>0$ such that
$$
\int_{\R^6} f(t)\, \log f(t) \, d\q d\x
$$
is bounded from above on bounded time intervals, for every $\beta_{sup}\ge \bar \beta$. In fact, 
\begin{eqnarray}
\label{eq:diffeq}
\frac{d}{dt} \int_{\R^6} f(t)\, \log f(t) \, d\q d\x &\le& C_a(t,f^0,L,R,\beta_{sup})
\\
\nonumber && + \ C_b(t,f^0,R,\beta_{sup}) \int_{\R^6}f(t) \log f(t) \, d\q d\x,
\end{eqnarray}
holds, where $C_a$ and $C_b$ are given by
$$
C_a:= \frac{2 \Phi(R;L,\tilde \beta_f)}{M(\beta_{sup})} C_{4}(f^0,t)\left(1+ \beta_{sup} \sqrt{1+L^2} + \left|\log \frac{2}{M(\beta_{sup})} \right| \right)
$$
$$
\hspace{-1cm}+\frac{2 \Lambda(2R)}{M(\beta_{sup})} C_{5}(f^0,t)\left(1+ \beta_{sup}  + \left|\log \frac{2}{M(\beta_{sup})} \right| \right) 
$$
$$
+  \frac{C_{6}(f^0 ,t)}{\beta_{sup}}(1+ \log \beta_{sup})
+ C_{7}(f^0 ,t) \log \left(1+\frac{1}{\sqrt{\beta_{sup}}}\right) 
$$
(with $C_{4},C_{5}, C_{6}$ and $C_{7}$ not depending on $\beta_{sup}$) and
$$
C_b:=  \frac{4}{\beta_{sup}}+\frac{2 \Lambda(2R)}{M(\beta_{sup})} + \frac{2 \Phi(R;L,\tilde \beta_f)}{M(\beta_{sup})}\:. 
$$ 
Moreover, we have that
$$
\lim_{\beta_{sup} \to \infty} C_a = \lim_{\beta_{sup} \to \infty} C_b =0.
$$
\end{Proposition}

The aim of the rest of the section is to provide a proof for Proposition \ref{entropy_estimate}. This requires a number of intermediate results. We start with the following useful inequality:
\begin{Proposition}
\label{ease}
Let $f$ be a solution to (\ref{apro}) with initial datum $f^0$ such that
$$
\int_{\R^6} (1+q^0+|\x|) f^0\, d\q d\x <\infty.
$$
Then there exists a positive constant $C_8(t,f^0)$ not depending on $\beta_{sup}$ such that the following estimate
$$
\int_{\R_{\x}^3} n  \, \log n \, d\x \le  \int_{\R^6} f  \, \log f \, d\x d\q \, + \, C_8(t,f^0)
$$
 holds.
\end{Proposition}
The proof of Proposition \ref{ease} is a direct consequence of the next auxiliary result.
\begin{Lemma}
\label{ll5}
Let $f(\x,\q) \ge 0$ be such that
$$
m_1:=\int_{\R^6} (1+|\x|+q^0) f\, d\q d\x <\infty.
$$
Then, the following estimates hold true:
\begin{enumerate}

\item 
$\displaystyle
\int_{\R_{\x}^3}n_f \sqrt{1+|\u_f|^2}  \log \left(n_f \sqrt{1+|\u_f|^2}\right)  \, d\x \le \int_{\R^6} f \log f \, d\x d\q + C_9(m_1).$

\item $\displaystyle
\int_{\R_{\x}^3} n_f \log n_f \, d\x \le \int_{\R_{\x}^3}n_f \sqrt{1+|\u_f|^2}  \log \left(n_f \sqrt{1+|\u_f|^2}\right)  \, d\x + C_{10}(m_1),
$
\end{enumerate}
where $C_9, C_{10}$ are positive constants not depending on $\beta_{sup}$.
\end{Lemma}
\begin{proof}
To deal with the first point let us introduce the auxiliary constant $K=1/ \int_{\R_{\q}^3} e^{-q^0}\, d\q$. Using \eqref{rel1} we may write
\[
\nonumber
\begin{split}
\displaystyle
\displaystyle \int_{\R_{\x}^3}n_f \sqrt{1+|\u_f|^2}&  \log \left(n_f \sqrt{1+|\u_f|^2}\right)  \, d\x \\
 \\
& \displaystyle  =\int_{\R_{\x}^3} \left(\int_{\R_{\q}^3} f(\q)\, d\q \right)  \log \displaystyle
\left(\int_{\R_{\p}^3} f(\p)\, d\p \right)  \,  d\x
\\ \\
& \displaystyle
= \int_{\R^6} f \log \left(K e^{-q^0} \int_{\R_{\p}^3} f(\p)\, d\p \right)\, d\q d\x + \int_{\R^6} (q^0 - \log K) f \, d\x d\q.
\end{split}
\]
Therefore, we have
\[\begin{split}
\displaystyle \int_{\R_{\x}^3}n_f \sqrt{1+|\u_f|^2}&  \log \left(n_f \sqrt{1+|\u_f|^2}\right)  \, d\x
\\ \\
& \displaystyle
= \int_{\R^6} f \log f \, d\q d\x + \int_{\R^6} (q^0 - \log K) f \, d\x d\q
\\ \\
& \displaystyle \quad
 - \int_{\R^6} \left\{f\log \left(\frac{f}{K e^{-q^0} \int_{\R_{\p}^3} f(\p)\, d\p}\right) +K e^{-q^0} \int_{\R_{\p}^3} f(\p) \, d\p - f\right\} \, d\q d\x
\\ \\
& \displaystyle \le \int_{\R^6} f \log f \, d\x d\q + |\log K| \int_{\R^6} (1+q^0) f\, d\q d\x,
\end{split}
\]
which proves the first point. To achieve the last inequality we used that 
\begin{equation}
\label{Jtrick}
x \log (x/y)+ y -x \ge 0, \quad \forall x,y \ge 0,
\end{equation}
which is a consequence of the convexity of $x \mapsto x \,\log x$.

To prove the second point, using that $x\,\log x$ is increasing for $x>1/e$, we notice that
\[
\begin{array}{rl}
\displaystyle
\int_{\R_{\x}^3} n_f  
\displaystyle \log n_f \, d\x & \le \displaystyle \int_{\{\x\in \R_{\x}^3/ n_f \ge 1/e\}} n_f \log n_f \, d\x
\\ \\
 & \displaystyle \le \int_{\{\x\in \R_{\x}^3/ n_f \ge 1/e\}} n_f \sqrt{1+|\u_f|^2}  \log (n_f \sqrt{1+|\u_f|^2})  \, d\x,
\end{array}
\]
leading to
\[
\begin{array}{rl}
  \displaystyle
\int_{\R_{\x}^3} n_f  
\displaystyle \log n_f \, d\x &
\\ \\
& \displaystyle
\leq\int_{\R_{\x}^3} n_f \sqrt{1+|\u_f|^2}  \log (n_f \sqrt{1+|\u_f|^2})  \, d\x
\\ \\
& \displaystyle \quad  - \int_{\{\x\in \R_{\x}^3/ n_f< 1/e\}} n_f \sqrt{1+|\u_f|^2}  \log (n_f \sqrt{1+|\u_f|^2})  \, d\x:=A-B.
\end{array}
\]
We now turn to the lower estimate on B:
\[
\begin{array}{rl}
\displaystyle
 B \ge & \displaystyle  \int_{\{\x\in \R_{\x}^3/ n_f< 1/e, n_f \sqrt{1+|\u_f|^2}<1\}} n_f \sqrt{1+|\u_f|^2}  \log (n_f \sqrt{1+|\u_f|^2})  \, d\x
\\ \\
\ge & \displaystyle  \int_{\{\x\in \R_{\x}^3/ n_f \sqrt{1+|\u_f|^2}<1\}} n_f \sqrt{1+|\u_f|^2}  \log (n_f \sqrt{1+|\u_f|^2})  \, d\x,
\end{array}
  \]
  so that
\[
\begin{array}{rl}
\displaystyle
 B \ge
& \displaystyle  \int_{\{\x\in \R_{\x}^3/ n_f \sqrt{1+|\u_f|^2}<e^{-|\x|}\}} n_f \sqrt{1+|\u_f|^2}  \log (n_f \sqrt{1+|\u_f|^2})  \, d\x
\\ \\
& \displaystyle + \int_{\{\x\in \R_{\x}^3/ e^{-|\x|}\le  n_f \sqrt{1+|\u_f|^2}\le 1\}} n_f \sqrt{1+|\u_f|^2}  \log (n_f \sqrt{1+|\u_f|^2})  \, d\x,
\end{array}
\]
and eventually
\[
\begin{array}{rl}
\displaystyle
 B \ge & \displaystyle  - \int_{\{\x\in \R_{\x}^3/ n_f \sqrt{1+|\u_f|^2}<e^{-|\x|}\}} \sqrt{n_f \sqrt{1+|\u_f|^2}}  \, d\x
\\ \\
& \displaystyle -\int_{\{\x\in \R_{\x}^3/ e^{-|\x|}\le  n_f \sqrt{1+|\u_f|^2}\le 1\}} |\x| n_f \sqrt{1+|\u_f|^2}   \, d\x.
\end{array}
\]
The fact that $x \log x \ge -\sqrt{x}$ for $x\in [0,1]$ was used to get the last inequality. Then, using \eqref{rel1},
$$
-B \le \int_{\R_{\x}^3} e^{-|\x|/2}\, d\x + \int_{\R^6} |\x| f \, d\x d\q,
$$
which yields the desired estimate.
\end{proof}

Let us compute now the time derivative of the $L \log L$ functional:
\begin{Lemma}
\label{L5}
Let $f$ be a solution of \eqref{apro}. Then, we have
 $$
\frac{d}{dt} \int_{\R^6} f \,  \log f \, d\q d\x \le \int_{\R^6} \tilde J[f] \log \tilde J[f]\, \frac{d\q}{q^0} d\x - \int_{\R^6}  J[f] \log  J[f]\, \frac{d\q}{q^0} d\x.
$$
Moreover, there exists some $\bar \beta_2>0$ such that the following estimate holds true for every $\beta_{sup}\ge \bar \beta_2$:
$$
\int_{\R^6} \tilde J[f] \log \tilde J[f]\, \frac{d\q}{q^0} d\x - \int_{\R^6}  J[f] \log  J[f]\, \frac{d\q}{q^0} d\x \le \mathcal{A+B+C+D}
$$
with
\begin{equation}
\nonumber
\begin{array}{rl}
\mathcal{A}:= & \displaystyle \int_{\R^6} n_f \log n_f \left( \frac{e^{-\tilde \beta_f (\tilde u_f)_\mu q^\mu}}{\tilde M(\tilde \beta_f)} -  \frac{e^{- \beta_f  (u_f)_\mu q^\mu}}{ M( \beta_f)}
\right)\, \frac{d\q}{q^0}d\x,
\\ \\
\mathcal{B}:= & \displaystyle \int_{\R^6} n_f \left(\frac{e^{- \beta_f  (u_f)_\mu q^\mu}}{ M( \beta_f)}\beta_f (u_f)_\mu q^\mu
- \frac{e^{-\tilde \beta_f (\tilde u_f)_\mu q^\mu}}{\tilde M(\tilde \beta_f)} \tilde \beta_f  (\tilde u_f)_\mu q^\mu 
\right)\, \frac{d\q}{q^0}d\x,
\\ \\
\mathcal{C}:= & \displaystyle \int_{\R^6} n_f \left(\frac{e^{- \beta_f  (u_f)_\mu q^\mu}}{ M( \beta_f)}\log M(\beta_f)
- \frac{e^{-\tilde \beta_f (\tilde u_f)_\mu q^\mu}}{\tilde M(\tilde \beta_f)}\log \tilde M(\tilde \beta_f) 
\right)\, \frac{d\q}{q^0}d\x
\end{array}
\end{equation}
and
\begin{eqnarray*}
\mathcal{D}&:=& \frac{2 \Phi(R;L,\tilde \beta_f)}{M(\beta_{sup})} \left( C_{11}(f^0,t) + \int_{\R_{\x}^3} n \log n \, d\x\right)\\
&&+ \frac{2 \Phi(R;L,\tilde \beta_f)}{M(\beta_{sup})} \left( \beta_{sup} \sqrt{1+L^2} + \left|\log \frac{2}{M(\beta_{sup})} \right| \right) \int_{\R_{\x}^3} n \, d\x
\end{eqnarray*}
where $\Phi$ is given by \eqref{phi} and $C_{11}$ is a positive constant not depending on $\beta_{sup}$.
\end{Lemma}
\begin{proof}
For simplicity, throughout the proof we will omit the subindex in $n_f, \u_f$ and $\beta_f $, which indicates the dependency of  $n, \beta$ and $\u$ on $f$. We compute
\[
\begin{array}{rl}
\displaystyle
\frac{d}{dt} \int_{\R^6} f \log f \, d\q d\x = &
\displaystyle
\int_{\R^6} \tilde J[f] (1+ \log f) \, \frac{d\q}{q^0} d\x - \int_{\R^6} f (1+ \log f) \, \frac{d\q}{q^0} d\x
\\ \\
= &
\displaystyle
-\int_{\R^6} \tilde J[f] \log \left( \frac{\tilde J[f]}{f}\right) + f - \tilde J[f] \, \frac{d\q}{q^0} d\x 
\\ \\
& \displaystyle
+ \int_{\R^6} \tilde J[f] \log \tilde J[f]\, \frac{d\q}{q^0} d\x - \int_{\R^6} f \log f \, \frac{d\q}{q^0} d\x.
\end{array}
\]
Using \eqref{Jtrick},
\[
\begin{array}{rl}
\displaystyle
\frac{d}{dt} \int_{\R^6} f \log f \, d\q d\x \le & \displaystyle
\int_{\R^6} \tilde J[f] \log \tilde J[f]\, \frac{d\q}{q^0} d\x - \int_{\R^6}  J[f] \log  J[f]\, \frac{d\q}{q^0} d\x
\\ \\
& \displaystyle + \int_{\R^6}  J[f] \log  J[f]\, \frac{d\q}{q^0} d\x - \int_{\R^6} f \log f \, \frac{d\q}{q^0} d\x.
\end{array}
\]
Thus, the first statement of the Lemma follows thanks  to Proposition~\ref{entropymin}.   

To derive the second statement, we start by expanding
\begin{equation}
\nonumber
\begin{array}{rl}
& \displaystyle \int_{\R^6} \tilde J[f]  \log \tilde J[f]\, \frac{d\q}{q^0} d\x = \int_{\R^6}\varphi(\q) \frac{n}{\tilde M(\tilde \beta)} e^{-\tilde \beta (\tilde u)_\mu q^\mu} \log \varphi(\q) \, \frac{d\q}{q^0} d\x 
\\ \\
&\quad \quad  \displaystyle +\int_{\R^6}\varphi(\q) \frac{n}{\tilde M(\tilde \beta)} e^{-\tilde \beta (\tilde u)_\mu q^\mu} \left( \log n - \log \tilde M(\tilde \beta) - \tilde \beta (\tilde u)_\mu q^\mu\right) \, \frac{d\q}{q^0} d\x :=A+B
\end{array}
\end{equation}
 
Next we split 
\begin{eqnarray*}
B&=& \int_{\R^6}\frac{n}{\tilde M(\tilde \beta)} e^{-\tilde \beta (\tilde u)_\mu q^\mu} \left( \log n - \log \tilde M(\tilde \beta) - \tilde \beta (\tilde u)_\mu q^\mu\right) \, \frac{d\q}{q^0} d\x\\
&&+ \int_{\R^6}(\varphi(\q)-1) \frac{n}{\tilde M(\tilde \beta)} e^{-\tilde \beta (\tilde u)_\mu q^\mu} \left( \log n - \log \tilde M(\tilde \beta) - \tilde \beta (\tilde u)_\mu q^\mu \right) \, \frac{d\q}{q^0} d\x:=B_1+B_2.
\end{eqnarray*}
We notice that 
$$
B_1 - \int_{\R^6}  J[f] \log  J[f]\, \frac{d\q}{q^0} d\x
$$
accounts for the terms $\mathcal{A+B+C}$ in the statement of the lemma. Thus, it only remains to give suitable bounds for $A$ and $B_2$ that we shall gather in the expression for $\mathcal{D}$. We clearly have
$$
|A|
 \le \frac{1}{e} \int_{\R_{\x}^3} \frac{n}{\tilde M(\beta_{sup})} \Phi(R;L,\tilde \beta)\, d\x.
$$
Let us estimate next the various terms composing $B_2$ in turn. First,
$$
\int_{\R^6} (1-\varphi(\q)) \frac{n}{\tilde M(\tilde \beta)} e^{-\tilde \beta (\tilde u)_\mu q^\mu}\beta (\tilde u)_\mu q^\mu\, \frac{d\q}{q^0}d\x \le \int_{\R_{\x}^3} \frac{n}{\tilde M(\beta_{sup})} \beta_{sup} \sqrt{1+L^2} \Phi(R;L,\tilde \beta)\, d\x.
$$
Next, we notice that $x\mapsto |x \log x|$ has a local maximum at $x=1/e$ (where it assumes the value $1/e$) and is increasing for $x\ge 1$. Hence, since $\tilde M$ is decreasing in $\beta$ and $\tilde \beta \leq \beta_{sup}$,
\begin{equation}
\label{logtrick}
\left| \frac{\log \tilde M(\tilde \beta)}{\tilde M(\tilde \beta)}\right| 
= \left| \frac{1}{\tilde M(\tilde \beta)} \log \frac{1}{\tilde M(\tilde \beta)}\right|
\le \left| \frac{1}{\tilde M( \beta_{sup})} \log \frac{1}{\tilde M( \beta_{sup})}\right|
\end{equation}
provided that $\beta_{sup}$ is large enough so that e.g. $\tilde M( \beta_{sup})\le 2/3$ -thus the rhs of \eqref{logtrick} exceeds $1/e$. 
Let us give sufficient conditions for this to happen. It suffices to find the range of $\beta$ for which we have $M( \beta_{sup})\le 2/3$. Now owing to \eqref{Mfar} we can find some $\bar \beta_2>\bar \beta_1$ (where $\bar \beta_1$ is defined in Lemma \ref{lcomp}) such that
\begin{equation}
\label{eq:sandwich}
\frac{1}{2} \left(\frac{2\pi}{\beta}\right)^{3/2}e^{-\beta} \le M(\beta) \le 2 \left(\frac{2\pi}{\beta}\right)^{3/2}e^{-\beta}, \quad \forall \beta \ge \bar \beta_2 .
\end{equation}
We readily see that the rhs of \eqref{eq:sandwich} is less than $1/3$ for e.g. $\beta \ge 3$. There is no loss of generality in assuming that $\bar \beta_2> 3$ and in this way \eqref{logtrick} is granted for $\beta_{sup} \ge \bar \beta_2$.

Once we made sure that \eqref{logtrick} holds for $\beta_{sup} \ge \bar \beta_2$ we use it in combination with \eqref{ratio_est}
to get
$$
\left|\int_{\R^6} (\varphi(\q)-1) n \frac{\log \tilde M(\tilde \beta)}{\tilde M(\tilde \beta)}  e^{-\tilde \beta (\tilde u)_\mu q^\mu} \, \frac{d\q}{q^0}d\x \right| \le \int_{\R_{\x}^3} \frac{2 n}{ M(\beta_{sup})} \left|\log \frac{2}{ M(\beta_{sup})} \right| \Phi(R;L,\tilde \beta)\, d\x.
$$
To bound the remaining term we have to distinguish between the case in which $n \log n \ge 0$ and the complementary one. Arguing like in the proof of the second estimate in Lemma \ref{ll5}, we find that
\[
\begin{array}{rl}
& \displaystyle \left| \int_{\R^6}  (1-\varphi(\q)) n \log n \frac{e^{-\tilde \beta (\tilde u)_\mu q^\mu}}{\tilde M(\tilde \beta)} \, \frac{d\q}{q^0}d\x \right|  
\\ \\
 &\quad \le \displaystyle \frac{\Phi(R;L,\tilde \beta)}{\tilde M(\beta_{sup})} \left(\int_{\R_{\x}^3} n \log n \, dx - \int_{\{\x \in\R_{\x}^3/n \log n <0\}} n \log n \, d\x \right)
\\ \\
& \quad \quad \displaystyle + \int_{\{\x\in \R_{\x}^3 /n \log n <0\}}(1-\varphi(\q)) |n \log n| \frac{e^{-\tilde \beta (\tilde u)_\mu q^\mu}}{\tilde M(\tilde \beta)} \, \frac{d\q}{q^0}d\x.
\end{array}
\]
Hence
\[
\begin{array}{rl}
& \displaystyle \left| \int_{\R^6}  (1-\varphi(\q)) n \log n \frac{e^{-\tilde \beta (\tilde u)_\mu q^\mu}}{\tilde M(\tilde \beta)} \, \frac{d\q}{q^0}d\x \right|  
\\ \\
&   \displaystyle \le  \frac{\Phi(R;L,\tilde \beta)}{\tilde M(\beta_{sup})} \left(\int_{\R_{\x}^3} n \log n \, d\x - 2  \int_{\{\x\in \R_{\x}^3/n \log n <0\}}  n \log n \, d\x \right)
\\ \\
&  \displaystyle \le  \frac{\Phi(R;L,\tilde \beta)}{\tilde M(\beta_{sup})} \left( 
\int_{\R_{\x}^3} n \log n \, d\x + 2 \int_{\{\x\in \R_{\x}^3/n <e^{-|\x|}\}} \sqrt{n} \, d\x + 2 \int_{\{\x \in \R_{\x}^3/e^{-|\x|}<n <1\}} |\x| n \, d\x \right),
\end{array}
\]
and finally
\[
\begin{array}{rl}
& \displaystyle \left| \int_{\R^6}  (1-\varphi(\q)) n \log n \frac{e^{-\tilde \beta (\tilde u)_\mu q^\mu}}{\tilde M(\tilde \beta)} \, \frac{d\q}{q^0}d\x \right|  
\\ \\
& \quad \displaystyle \le  \frac{\Phi(R;L,\tilde \beta)}{\tilde M(\beta_{sup})} \left( 
\int_{\R_{\x}^3} n \log n \, d\x + 2 \int_{\R_{\x}^3} e^{-|\x|} \, d\x + 2 \int_{\R^6} |\x| f\, d\x d\q
\right)
\\ \\
& \quad \displaystyle =  \frac{\Phi(R;L,\tilde \beta)}{\tilde M(\beta_{sup})} \left( 
\int_{\R_{\x}^3} n \log n \, d\x + C_{11}(t,f^0).
\right)
\end{array}
\]
Now we note that we may replace all factors of $\tilde M(\beta_{sup})^{-1}$ in the above estimates by $2/M(\beta_{sup})$ using Lemma \ref{lcomp} provided that $\beta_{sup} \ge \bar \beta_1$. Finally, the second claim of the Lemma is an easy consequence of our estimates so far and Proposition \ref{ease}.
\end{proof}
To derive Proposition \ref{entropy_estimate} it only remains to estimate $\mathcal{A}, \, \mathcal{B}$ and $\mathcal{C}$ in turn.
\begin{Lemma}
\label{lm:9}
There is  $\bar \beta_3>0$ such that the following property holds true: there exists a positive constant $C_{12}(f^0,t)$ (not depending on $ \beta_{sup}$) such that
$$
\mathcal{A} \le  \left(\int_{\R^6} f \log f \, dxd\q +C_{12}(t,f^0)\right) \left(  \frac{4}{\beta_{sup}}+\frac{2 \Lambda(2R)}{M(\beta_{sup})} \right),
$$
for every $\beta_{sup}\ge \bar \beta_3$,
where $\Lambda$ is given \eqref{lambda}.
\end{Lemma}
\begin{proof}
As in the proof of the previous result, will omit the subindex in $n_f, \u_f$ and $\beta_f $. We use Lemma \ref{Jmoments}-({\ref{id:7}}) to write 
$$
\mathcal{A} = \int_{\R_{\x}^3}  n \log n \left(\frac{M}{\tilde M}(\tilde \beta) \frac{K_1}{K_2}(\tilde \beta)-\frac{K_1}{K_2}(\beta)\right)\, d\x
=\int_{\R_{\x}^3} (X+Y) n \log n \, d\x
$$
with
$$
X:= \frac{K_1}{K_2}(\tilde \beta) - \frac{K_1}{K_2}(\beta) \quad \mbox{and}\quad Y := \frac{K_1}{K_2}(\tilde \beta) \left(\frac{M(\tilde \beta)-\tilde M(\tilde \beta)}{\tilde M (\tilde \beta)} \right). 
$$
Let us start by estimating $X$. Note that $X=0$ whenever $\beta_{inf} \le \beta\le \beta_{sup}$. If $\beta<\beta_{inf}$, then 
$$X \le 2 |\beta_{inf}-\beta|\le 4 \beta_{inf}$$ due to \eqref{kanear}. For $\beta >\beta_{sup}$ we make use of \eqref{karat}; there 
exists some $\bar \beta_3 \ge \bar \beta_2$ (recall that $\bar \beta_2$ is defined in Lemma \ref{L5}) such that 
\begin{equation}
\label{eq:sandwichK}
\left|\frac{K_1}{K_2}(\beta)-1 \right| \le \frac{2}{\beta}
\end{equation}
holds for every $\beta \ge \bar \beta_3$. Then clearly $|X| \le 4/\beta_{sup}$ for $\beta >\beta_{sup}$ as $K_1/K_2$ is strictly increasing from zero to one (see Lemma \ref{lm:BCNS}). Overall, the following estimate holds:
$$
|X| \le \frac{4}{\beta_{sup}}.
$$
To estimate $Y$ we use Lemma \ref{lcomp}:
$$
Y \le \frac{\Lambda(R)}{\tilde M(\tilde \beta)} \le \frac{2 \Lambda(2R)}{M(\tilde \beta)} \le \frac{2 \Lambda(R)}{M(\beta_{sup})},
$$
which holds for $\beta_{sup}\ge \bar \beta_1$. 

As $n \, \log \, n$ need not have a sign, we argue as in the proof of Lemma \ref{ll5}. On one hand,
\[
\begin{array}{rl}
& \displaystyle \int_{\R_{\x}^3} (X \chi_{\{X\ge 0\}} + Y) n \log n \, d\x 
\\ \\
 &\le  \displaystyle \int_{\R_{\x}^3} n \sqrt{1+|\u|^2}  \log (n \sqrt{1+|\u|^2})  (X \chi_{\{X\ge 0\}} + Y) \, d\x 
\\ \\
& \displaystyle \quad - \int_{\{x\in \R_{\x}^3/ n< 1/e\}} n \sqrt{1+|\u|^2}  \log (n \sqrt{1+|\u|^2})(X \chi_{\{X\ge 0\}} + Y)  \, d\x.
\end{array}
\]
This leads to
\[
\begin{array}{rl}
& \displaystyle \int_{\R_{\x}^3} (X \chi_{\{X\ge 0\}} + Y) n \log n \, d\x 
\\ \\
&\le \displaystyle \left\{4 \beta_{inf} + \frac{2 \Lambda(2R)}{M(\beta_{sup})} \right\}\left(\int_{\R^6} \, f \log f dxd\q +C_8(t,f^0)\right)
\\ \\
& \quad + \displaystyle \left\{4 \beta_{inf} + \frac{2 \Lambda(2R)}{M(\beta_{sup})} \right\}\left(\int_{\R^6} \, |\x| f d\x d\q + \int_{\R_{\x}^3} e^{-|\x|/2} \, d\x \right),
\end{array}
\]
where we used Proposition \ref{ease}.

On the other hand,
\[
\begin{split}
  &\int_{\R_{\x}^3} X \chi_{\{X< 0\}}  n \log n \, d\x \le \int_{\{\x\in \R_{\x}^3/n <1/e\}} X \chi_{\{X< 0\}}  n \log n \, d\x\\
  & \le \frac{4}{\beta_{sup}} \int_{\{\x\in \R_{\x}^3/n <1/e\}} | n \log n| \, d\x\\
&\le \frac{4}{\beta_{sup}} \left( \int_{\{\x\in \R_{\x}^3/n < e^{-|\x|}\}} e^{-|\x|/2} \, d\x + \int_{\{\x \in \R_{\x}^3/e^{-|\x|}<n <1\}} |\x| n \, d\x \right).
\end{split}
\]
Collecting all estimates  concludes the proof.
\end{proof}

\noindent
In order to estimate $\mathcal{B}$ and $\mathcal{C}$, we introduce the following notations:
$$
T:=\{\x\in \R_{\x}^3 \, / \, \beta_f > \beta_{sup}\} \quad \text{and} \quad T^c=\{\x\in \R_{\x}^3 \, / \, \beta_f \le \beta_{sup}\}.
$$
\begin{Lemma}
\label{l:10}
There exists some $\bar \beta_4>0$ such that the following estimates 
\begin{enumerate}

\item $\displaystyle \int_{\R_\q^3 \times T^c} n_f  \left(\frac{e^{- \beta_f  (u_f)_\mu q^\mu}}{ M( \beta_f)}\beta_f (u_f)_\mu q^\mu
- \frac{e^{-\tilde \beta_f (\tilde u_f)_\mu q^\mu}}{\tilde M(\tilde \beta_f)} \tilde \beta_f  (\tilde u_f)_\mu q^\mu 
\right)\, \frac{d\q}{q^0}d\x$

$\displaystyle \le  2 \frac{\Lambda(2R) \beta_{sup}}{M(\beta_{sup})} \int_{\R^6} f\, d\q d\x,
$
\item $\displaystyle \int_{\R_\q^3 \times T^c} n_f  \left(\frac{e^{- \beta_f   (u_f )_\mu q^\mu}}{ M( \beta_f  )}\log M(\beta_f )
- \frac{e^{-\tilde \beta_f  (\tilde u_f )_\mu q^\mu}}{\tilde M(\tilde \beta_f )}\log \tilde M(\tilde \beta_f ) 
\right)\, \frac{d\q}{q_0}d\x $

$\displaystyle \le C_{13}(t,f^0,\beta_{sup}),$ 
\end{enumerate}
hold for every $\beta_{sup} \ge \bar \beta_4$, where 
$$
C_{13}(t,f^0,\beta_{sup}):=  \frac{2 \Lambda (2R)}{M(\beta_{sup})} \left(1+\left| \log \frac{2}{M(\beta_{sup})}\right| \right) \int_{\R^6} f\, d\q d\x 
$$
$$
+  \int_{\{\x \in \R_{\x}^3/\beta_f(\x)  <\beta_{inf}\}} \frac{n }{2}\beta_f \left(\log 16 \pi - 3 \log \beta_f  \right) \, d\x
$$
goes to 0 as $\beta_{sup} \to \infty$.
\end{Lemma}
\begin{proof}
For simplicity we omit along the proof the dependence of $f$ of $n$, $\u$, and $\beta$. Note that
\begin{equation}
\nonumber
\displaystyle \mathcal{B}  =  \displaystyle \int_{\R_{\x}^3} n  \beta \, d\x-  \int_{\R_{\x}^3} n  \tilde \beta  \frac{M}{\tilde M}(\tilde \beta ) \, d\x
 \displaystyle = \int_{\R_{\x}^3} n  (\beta - \tilde \beta )\, d\x +  \int_{\R_{\x}^3} n  \tilde \beta \left(1-\frac{M}{\tilde M}(\tilde \beta ) \right)\, d\x.
\end{equation}
We assume in the sequel that $\beta_{sup} \ge \bar \beta_1$. Then clearly
$$
\left|\int_{T^c} n  \tilde \beta \left(1-\frac{M}{\tilde M}(\tilde \beta ) \right)\, d\x \right|\le 2 \frac{\Lambda(2R) \beta_{sup}}{M(\beta_{sup})} \int_{\R^6} f\, d\q d\x ,
$$
 thanks to Lemma \ref{lcomp} and
$$
\int_{T^c} n  (\beta - \tilde \beta )\, d\x \le 0.
$$
This proves the first estimate. To get the second one, we write
\begin{equation}
\nonumber
\begin{array}{rl}
\displaystyle  \mathcal{C}:= &\displaystyle
\int_{\R_{\x}^3} n   \left(\frac{K_1}{K_2}(\beta )\log M(\beta )
- \frac{M}{\tilde M}(\tilde \beta ) \frac{K_1}{K_2}(\tilde \beta )\log \tilde M(\tilde \beta ) 
\right)\, d\x
\\ \\
= &\displaystyle
 \int_{\R_{\x}^3} n   \left(\frac{K_1}{K_2}(\beta )\log M(\beta )
- \frac{K_1}{K_2}(\tilde \beta )\log \tilde M(\tilde \beta ) 
\right)\, d\x
\\ \\
& \displaystyle
+ \int_{\R_{\x}^3} n   \frac{K_1}{K_2}(\tilde \beta )\log \tilde M(\tilde \beta )
\left(1
- \frac{M}{\tilde M}(\tilde \beta ) 
\right)\, d\x.
\end{array}
\end{equation}
To bound the second term above, notice that we only need to care about those large values of $\tilde \beta $ for which $\log \tilde M(\tilde \beta )<0$, otherwise that integral is non-positive. Then,  
 we have 
$$
0\ge \log \tilde M(\tilde \beta ) \ge \log \tilde M( \beta_{sup})= -\log (1/\tilde M( \beta_{sup})) \ge - \log (2/M( \beta_{sup}),
$$
for $\tilde \beta \ge \bar \beta_1$, where we used Lemma \ref{lcomp}. This means that 
$$
 |\log \tilde M(\tilde \beta )|\le \left| \log \frac{2}{M(\beta_{sup})}\right| \quad \mbox{in}\ T^c.
$$
 Then, using again Lemma \ref{lcomp}, we find
\begin{equation}
\nonumber
\displaystyle
 \int_{T^c} n    \displaystyle
 \frac{K_1}{K_2}(\tilde \beta )\log \tilde M(\tilde \beta )
\left(1
- \frac{M}{\tilde M}(\tilde \beta ) 
\right)\, d\x
 \le \frac{2 \Lambda(2R) }{M(\beta_{sup})} \left| \log \frac{2}{M(\beta_{sup})}\right| \int_{\R^6} f\, d\q d\x.
\end{equation}
Therefore we focus on the first term in the above formula for $\mathcal{C}$. Let
\begin{equation}
\nonumber
\begin{split}
\displaystyle \int_{T^c}  \displaystyle n   &\left(\frac{K_1}{K_2}(\beta )\log M(\beta )
- \frac{K_1}{K_2}(\tilde \beta )\log \tilde M(\tilde \beta ) 
\right)\, d\x
\\
&=   \displaystyle \int_{T^c} n  \log \tilde M(\tilde \beta ) \left(\frac{K_1}{K_2}(\beta )
- \frac{K_1}{K_2}(\tilde \beta ) 
\right)\, d\x
\\
& \quad +  \displaystyle \int_{T^c} n  \frac{K_1}{K_2}(\beta ) \left(\log M(\beta ) - \log \tilde M(\beta ) \right) \, d\x
\\
& \quad + \displaystyle \int_{T^c} n  \frac{K_1}{K_2}(\beta ) \left(\log \tilde M(\beta ) - \log \tilde M(\tilde \beta ) \right) \, d\x
\\ &:= I + II + III.
\end{split}
\end{equation}
{To handle $I$, we note that it necessarily vanishes unless $\beta <\beta_{inf}$. In that case, the integrand can only be positive where $\tilde M(\tilde \beta )<1$.
  But this does not take place provided that $\beta_{inf}$ is small enough; recall here that $\beta_{sup}= 1/\beta_{inf}$. This can be shown by a continuity argument, that we outline next.

  We recall that $\tilde M=\tilde M(\beta,R)$ actually depends on $R$ (let us forget about the constraint \eqref{eq:fix} for the moment). We have that $R\mapsto \tilde M(\beta;R)$ is increasing for fixed $\beta$. Hence $1<\tilde M(0,R=1)\le \tilde M(0,R)$. Then, by the continuity of $\tilde M$ around $(\beta=0,R=1)$, there is some $\beta^*(R)$ such that $I\le 0$, for any $\beta_{sup}>\beta^*(R)$. Moreover, $\beta^*(R)$ is decreasing in $R$; recalling now from \eqref{eq:fix}, that we connected $R$ and $\beta_{sup}$ through $R=\beta_{sup}^2$, there exists $\beta_4$ such that if $\beta_{sup}\geq \beta_4$, then $\beta_{sup}>\beta^*(\beta_{sup}^2)$.}

Now to deal with $II$ we use that $\log(1+x) \le x$ for $x\ge 0$, so that
$$
\log M(\beta )- \log \tilde M(\beta ) = \log \left(1+\frac{M(\beta )-\tilde M(\beta )}{\tilde M(\beta )} \right) \le \frac{\Lambda (2R)}{\tilde M(\beta )} \le 2 \frac{\Lambda (2R)}{M(\beta_{sup})} 
$$
and hence
$$
II\le \frac{2 \Lambda (2R)}{M(\beta_{sup})} \int_{\R^6} f\, d\q \, d\x.
$$
{Finally, the contribution of $III$ is non-negative only when $ \beta  \le \beta_{inf}$. In that case, provided that $\beta_{sup}\ge \beta^*$ is small enough, so that $\tilde M(\beta_{inf})\ge 1$, 
 we may resort to \eqref{Mnear} and find some $\bar \beta_4 \ge \max (\bar \beta_3,\beta^*)$ (recall that $\bar \beta_3$ is defined in Lemma \ref{lm:9}) such that
$$
\log \frac{\tilde M(\beta )}{\tilde M(\beta_{inf})} \le \log \tilde M(\beta ) \le \log M(\beta )\le \log 16 \pi - 3 \log \beta  
$$
holds, for every $\beta \le 1/\bar \beta_4$. We also use that $\frac{K_1}{K_2}(\beta) \le \beta/2$ to get to
$$
III \le \int_{\{\x \in \R_{\x}^3: \beta(\x)  <\beta_{inf}\}} \frac{n }{2}\beta \left(\log 16 \pi - 3 \log \beta\right) \, d\x.
$$
This is clearly finite and converges to zero when $\beta_{sup}$ diverges.
}
\end{proof}

It only remains to estimate the contribution of $\mathcal{B+C}$ over the residual set $\R_\q^3 \times T$. For that aim, we use Lemma \ref{Jmoments}, \emph{\ref{id:3}} and Remark \ref{varios1}, \emph{\ref{v:2}} to decompose
$$
\mathcal{B+C}= \int_{\R_{\x}^3} n_f  \left\{ \beta_f  + \frac{K_1}{K_2}(\beta_f )\log M(\beta_f ) - \frac{M}{\tilde M}(\tilde \beta_f ) \left( \tilde \beta_f  + \frac{K_1}{K_2}(\tilde \beta_f ) \log \tilde M(\tilde \beta_f )\right)\right\} \, d\x
$$
and then the contribution over the residual set reads
\begin{equation} \nonumber
\begin{split}
(\mathcal{B+C})_T &\hspace{0.1cm}= \int_{T} n_f  \left(\beta_f  + \frac{K_1}{K_2}(\beta_f ) \log M(\beta_f ) - \tilde \beta_f - \frac{K_1}{K_2}(\tilde \beta_f ) \log \tilde M(\tilde \beta_f ) \right) \, d\x
\\ &\quad +\int_{T} n_f  \tilde \beta_f  \left(1-\frac{M}{\tilde M}(\tilde \beta_f ) \right)\, d\x + \int_{T} n_f  \frac{K_1}{K_2}(\tilde \beta_f ) \log \tilde M(\tilde \beta_f ) \left(1-\frac{M}{\tilde M}(\tilde \beta_f ) \right) \, d\x
\\ &:= BC_1 + BC_2 + BC_3.
\end{split}
\end{equation}

\begin{Lemma}
\label{lm:17}
There is some $\bar \beta_5>0$ such that the following estimate 
\begin{enumerate}
\item $\displaystyle
BC_1 \le \frac{2\Lambda(2R)}{M(\beta_{sup})} \int_{\R^6} f\, d\q d\x + 9 \int_T   \frac{n}{\beta_{sup}} \log \beta_{sup}\, d\x
$

$\displaystyle+ 2 \int_T n \left( 1-\frac{3}{\beta_{sup}} \right) \log \left(1+\frac{1}{\sqrt{\beta_{sup}}}\right)\, d\x\:,
$

where the second and third terms on the right hand side vanish when $\beta_{sup} \to \infty$.

\item $\displaystyle
BC_2 \le \int_{\R_{\x}^3} n_f \beta_{sup} \frac{2 \Lambda (2R)}{M(\beta_{sup})}\, d\x\:. 
$

\item $\displaystyle
BC_3 \le \int_{\R_{\x}^3} n_f  \frac{2 \Lambda (2R)}{M(\beta_{sup})} \left| \log \frac{2}{M(\beta_{sup})}\right|\, d\x\:
$

\end{enumerate}
holds for every $\beta_{sup} \ge \bar \beta_5$.
\end{Lemma}
\begin{proof}
The estimate for $BC_2$ follows like in Lemma \ref{lcomp} and that for $BC_3$ is a consequence of \eqref{logtrick}. Let us address the estimate for $BC_1$. 
We split
$$
BC_1 = \int_{T} n  \left\{ \beta  + \frac{K_1}{K_2}(\beta )\log M(\beta ) -  \left( \tilde \beta  + \frac{K_1}{K_2}(\tilde \beta ) \log M(\tilde \beta )\right)\right\} \, d\x
$$
$$
+ \int_{T} n  \frac{K_1}{K_2}(\tilde \beta ) \left(\log M(\tilde \beta) - \log \tilde M(\tilde \beta ) \right)\, d\x:= \Upsilon_{1} + \Upsilon_{2},
$$
where we omit the subindex $f$ in the dependence of $n$, $\u$ and $\beta$ for simplicity. We get a bound on $\Upsilon_{2}$ in the same way as we did with $II$ in the proof of Lemma \ref{l:10}. Hence
$$
\Upsilon_{2}\le  \frac{2 \Lambda (2R)}{M(\beta_{sup})} \int_{\R^6} f\, d\q \, d\x.
$$
We study next $\Upsilon_{1}$ for $\beta_{sup}$ so big that we may use \eqref{karat} and \eqref{Mfar} in such a way that the error terms in those formulas represent faithfully the corrections needed when replacing the functions by the leading term in the expansion. Note that, after \eqref{Mfar}, 
$$
\log M(\beta) \sim -\beta  -\frac{3}{2} \log \beta  + \log \left( (2 \pi)^{3/2} + o(1/\beta )\right),
$$
for $\x\in T$, if $\beta_{sup}$ is large enough. Replacing these asymptotic expansions into $\Upsilon_{1}$ we meet many cancellations, so that 
\begin{equation}
\nonumber
\begin{split}
\displaystyle \Upsilon_{1} & =  
%+  
\int_{T} n  \left(\beta_{sup} o(1/\beta_{sup}) -\beta  o(1/\beta ) \right)\, d\x
\\
 & \quad \displaystyle + \frac{3}{2} \int_{T}  n  \left\{\log \beta  \left(\frac{3}{2 \beta } + o(1/\beta ) \right)-\log \beta_{sup} \left(\frac{3}{2 \beta_{sup}} + o(1/\beta_{sup}) \right) \right\}\, d\x
 \\
 &  \quad \displaystyle + \frac{3}{2} \int_{T}  n  \log (\beta_{sup}/\beta ) \, d\x
  \\
 & \quad  \displaystyle + \int_{T}  n  \log (2\pi)^{3/2} \left(\frac{3}{2 \beta_{sup}} - \frac{3}{2 \beta } + o(1/\beta_{sup}) + o(1/\beta )\right)\, d\x
  \\
 &  \quad \displaystyle + \int_{T}  n  \left\{ \log (1+o(1/\beta )) \left(1-\frac{3}{2 \beta }+o(1/\beta )\right)\right.
   \\
 &  \quad \quad \displaystyle 
 \left.  - \log (1+o(1/\beta_{sup})) \left(1-\frac{3}{2 \beta_{sup}}+o(1/\beta_{sup})\right)\right\} \, d\x
 \\ & := I+\cdots +V.
\end{split}
\end{equation}
On one hand, we have that $II, \, III \le 0$. On the other hand, we can find some $\bar \beta_5 \ge \bar \beta_4$, being $\bar \beta_4$ defined in Lemma \ref{l:10}, such that the following estimates hold for $\beta_{sup}\ge \bar \beta_5$:
$$
I\le \frac{3}{2} \int_T 2 n \frac{3}{\beta_{sup}} \log \beta_{sup}\, d\x
$$
and
$$
IV+V\le 2 \int_T n \left( 1-\frac{3}{\beta_{sup}} \right) \log \left(1+\frac{1}{\sqrt{\beta_{sup}}}\right)\, d\x \:.
$$ 
Both bounds vanish in the limit $\beta_{sup} \to \infty$. This concludes the proof of the lemma.
\end{proof}
\noindent
{\em Proof of Proposition \ref{entropy_estimate}:} we work with $\beta_{sup} \ge \bar \beta:= \bar \beta_5$, which ensures us that Lemmas \ref{ll5}-\ref{lm:17} and Proposition \ref{ease} hold true. Gathering all those estimates and performing some straightforward majorizations we arrive to the differential inequality (\ref{eq:diffeq}). This entails an upper bound for the relative entropy on finite time  intervals. To conclude, we may show that the constants $C_a, C_b$ vanish in the limit $\beta_{sup} \to \infty$ by a cursory inspection, after taking into account \eqref{eq:fix} and the expansion \eqref{Mfar}.

%%%%%%%%%%%%%%%%%%
\section{Passing to the limit}
\label{sec:5}
In this section we study the limit behavior of the approximations constructed in section \ref{sec:4} as $\beta_{sup} \to \infty$ (recall the constraints \eqref{eq:fix} that reduce all the regularizing parameters to a single one). In order to ease the notation, we set $\epsilon:=1/\beta$ during the rest of the section and we study instead the limit $\epsilon \to 0$ of the sequence of approximations $\{f_\epsilon\}_\epsilon$. Note that in the limit $\epsilon \to 0$ the restriction on the support of $f_0$ in $\q$ disappears.
\begin{Lemma}
\label{rhscontrol}
Let $f_\eps$ be a solution to (\ref{apro}) with initial datum $f^0\ge 0$ such that
$$
\int_{\R^6} (1+q^0+|\x|+\log f^0) f^0\, d\q d\x <\infty.
$$
Then 
$$
\int_{\R^6} f_\eps(t) \, |\log f_\eps(t)| \,  d\x d\q, \quad \int_{\R_{\x}^3} \int_{\R_{\q}^3} J_{f_\eps} \log J_{f_\eps}\ d\x\frac{d\q}{q^0}, \quad \mbox{and} \quad \int_{\R_{\x}^3} \int_{\R_{\q}^3} J_{f_\eps}\left|\log J_{f_\eps}\right|\ d\x\frac{d\q}{q^0}
$$
are uniformly bounded in $\eps$ on bounded time intervals. 
\end{Lemma}
\begin{proof}
The first point is a consequence of the following inequality (see e.g. \cite{Perthame})
\begin{equation}
\label{Huanho}
g(\x,\q) |\log g(\x,\q)| \le g(\x,\q) \log g(\x,\q) + (|\x|+|\q|)g(\x,\q) + \frac{1}{e} e^{-\frac{|\x|+|\q|}{4}}.
\end{equation}
The second point follows from Proposition \ref{optimality} and Proposition \ref{entropymin}. Then the third point is worked out as the first one.
\end{proof}

\subsection{Momenta averaging}
We have shown that under mild assumptions on the initial datum we have that both ${f_\eps}/q^0$ and $J_{f_\eps}/q^0$ belong to the $L \log L (\R^6)$ class. 

The combination with the moment estimates with respect to the measure $\frac{d\q}{q^0}$ is  enough to ensure that momentum averages of solutions are strongly compact in $L^1(\frac{d\q}{q^0})$.

We first state a general result for averaging lemmas in the relativistic context. This theorem is a straightforward extension of the classical $L^1$ compactness that is used in collisional models; see in particular \cite{G-SR}.
\begin{Theorem}
\label{t:1}
Let $f_\eps$ and $g_\eps$ be two sequences of functions uniformly bounded in
$L^\infty([0,\ T],$ $ L^1(\R^{6}))$ in $\eps$, and solutions to the following kinetic
equations 
\[
\partial_t f_\eps+\frac{\q}{q^0}\cdot\nabla_{\x} f_\eps=g_\eps.
\]
We assume moreover that
\[
\sup_\eps \sup_{t\in [0,\ T]} \int_{\R^{6}} |f_\eps| (|\x|+|\q|+{|\log
f_\eps|})\,d\x\,d\q<\infty. 
\]
Then, the moment $\int_{\R^3} f_\eps\,\psi(\q)\,d\q$ is strongly compact in $L^1([0,\ T]\times \R^3)$, for any test function $\psi(\q)\in C^\infty(\R^d)$ such that
\[
\frac{|\psi(\q)|}{|\q|}\longrightarrow 0,\quad as \ |\q|\rightarrow +\infty.
\]
\end{Theorem}

\begin{proof}
Let
$\Phi_R$, $\psi_V$  be two truncation functions satisfying 
\[
\begin{split}
&\Phi_F(\xi)=\xi\quad if\ |\xi|\leq
F,\qquad \Phi_F(\xi)=2F\,\frac{\xi}{|\xi|}\quad if\ |\xi|\geq
2F,\qquad |\Phi_F'(\xi)|\leq 1\quad\forall \xi,\\ 
&\psi_V(q)=1\quad if\ |q|\leq V,\qquad \psi_V(q)=0\hspace{1.3cm} if\ |q|\geq
2V,\qquad |\psi_V(\xi)|\leq 1\quad\forall \xi.\\ 
\end{split}
\]
Then, applying the theory of renormalized solutions from \cite{DL}, one first checks that $f_\eps^F=\Phi_F(f_\eps)$ satisfies the equation
\[
\partial_t f_\eps^F+\frac{\q}{q^0}\cdot \nabla_x f_\eps^F=h_\eps^F=g_\eps\,\Phi'_F(f_\eps),
\]
with hence $h_\eps$ uniformly bounded in $L^1$ in $\eps$ and $F$.

First we obtain 
some regularity, uniform in $\eps$, of the moment
$\int f_\eps^F\,\psi(q)\,\psi_V(q)\,dq$. For this, notice that $\psi\,\psi_V\in
C^\infty_c$ and that $\q/q^0$ satisfies the usual assumption for
averaging lemmas on any compact support (and so in particular on the
support of $\psi_V$) as the transform $\q\rightarrow \q/q^0$ is one to one with bounded Jacobian. Hence
\[
\exists C_V,\ \forall\ \xi\in\R^d,\quad
|\{q\in \mbox{supp}\,\psi_V\,,\;|\q/q^0-\xi|\leq \eta\}|\leq
 C_V\,\eta.
\] 
In addition $f_\eps^F$ is of course bounded in $L^2$ uniformly in $\eps$, as it is truncated by the definition of $\Phi_F$ and by Cauchy-Schwartz
\[
\sup_\eps \|f_\eps^F\|_{L^2([0,\ T]\times \R^6)}\leq T\,\sqrt{2F}\,\sup_\eps \|f_\eps\|_{L^\infty([0,\ T],\ L^1(\R^6))}^{1/2}.
\]

One may hence apply the result of \cite{DL} for instance
and obtain that for any compact set $\Omega\subset \R_+\times \R^3$ and for some constant
$C_{F,V,\Omega}$, depending on $\Omega$, 
and blowing up with $F$ and $V$ but independent of $\eps$ 
\begin{equation}
\left\|\int_{\R^3} f_\eps^F\,\psi(\q)\,\psi_V(\q)\,d\q\right\|_{W^{s,p}_{t,x}(\Omega)}\leq C_{F,V,\Omega},
\label{avleml2}\end{equation}
for some $s>0$ and $p\in (1,\;2)$, which could be computed explicitly
but whose expressions are unimportant here. 

Now one simply writes
\[\begin{split}
\int_{\R^3} f_n\,\psi(\q)\,d\q=&\int_{\R^3}
f_\eps^F\,\psi(\q)\,(1-\psi_V(\q))\,d\q
+\int_{\R^3} (f_\eps-f_\eps^F)\,\psi(\q)\,\psi_V(\q)\,d\q\\
&+\int_{\R^3} f_\eps^F\,\psi(\q)\,\psi_V(\q)\,d\q.
\end{split}\] 
The last term is of course locally compact in $L^1_{t,x}$, for $F$ and $V$
fixed by \eqref{avleml2}.

The first two terms are small in $L^1$ as $F$ and $V$ are large, 
independently of 
$\eps$,
provided $\psi(2V)/\log F$ is small,  since
\[
\left\|\int_{\R^3}
f_\eps^F\,\psi(\q)\,(1-\psi_V(\q))\,d\q\right\|_{L^1}
\leq \frac{\psi(V)}{V}\int_{\R^{6}} |\q|\,f_\eps\,d\q,
\]
and
\[\begin{split}
\left\|\int_{\R^3}
(f_\eps-f_\eps^F)\,\psi(\q)\,\psi_V(\q)\,d\q\right\|_{L^1}&\leq \psi(2V)\int_{\R^{6}}
|f_\eps-f_\eps^F|\,d\q\,d\x\\
& \leq \frac{\psi(2V)}{\log F}\,\int_{\R^{6}}
|f_\eps|\,|\log f_\eps|\,d\q\,d\x.
\end{split}\]
Hence one deduces first that $\int f_\eps\,\psi(\q)\,d\q$ is locally
compact in $L^1(\R_+\times\R^3)$. 

To conclude and get the compactness over the whole $L^1$, it enough to
control a moment in $x$. For that we recall the
duality inequality
\[
a\,\psi(b)\leq |b|+\Psi^*(a),\ \forall a\in \R,\ b\in \R^d,
\]
where $\Psi^*$ is the convex conjugate function of $\Psi(a)=\inf\{|b|,\
|\psi(b)|\leq a\}$. Note that as $|\psi(b)|/|b|\rightarrow 0$ as
$|b|\rightarrow +\infty$ then everything is well defined and one has
of course that $\Psi^*(a)\rightarrow \infty$ as $a\rightarrow \infty$.
Therefore defining $m(r)=\sup\{a,\ \Psi^*(a)\leq r\}$, one has
\[
m(r)\,\psi(b)\leq r+|b|,\quad m(r)\rightarrow \infty\ \mbox{as}\
a\rightarrow \infty. 
\]
Finally just observe that
\[
\int_{\R^{2d}} m(|\x|)\,|f_\eps|\,\psi(\q)\,d\x\,d\q
\leq \int_{\R^{2d}} (|\q|+|\x|)\,|f_\eps|\,d\x\,d\q, 
\]
and is, therefore, bounded which concludes the proof.
\end{proof}

%%%%%%%%%%%%%%%%%%
\subsection{Passing to the limit}
Remark that from the uniform bounds provided by Lemma \ref{rhscontrol} we have the following result.
\begin{Lemma}
\label{l:13}
Let $f^0\ge 0$ such that
$$
\int_{\R^6} (1+q^0+|\x|+\log f^0) f^0\, d\q d\x <\infty.
$$
Let $T>0$, and consider $f_\eps$ the solution to \eqref{apro} in $[0,T]\times \R^6$ with initial datum $f^0$. Then, the following statements are verified
\begin{enumerate}
\item The family $\{f_\eps\}$ is weakly compact in $L^1([0,T]\times \R^6)$.
\item The family $\tilde J[f_\eps]$ is weakly compact in $L^1([0,T]\times \R_{\x}^3,L^1(\R_{\q}^3,\frac{d\q}{q^0}))$.
\end{enumerate}
\end{Lemma}
We may hence extract subsequences
\begin{Definition}
Let $\{\epsilon_n\}$ be a \emph{given} subsequence such that $\{f_{\eps_n}\}$ and $\tilde J[f_{\eps_n}]$ converge in the sense specified in Lemma \ref{l:13}. We set:
 $$
 f:=\lim_{n\to \infty} f_{\eps_n} \quad \text{and} \quad \mathcal{J}:=\lim_{n\to \infty} \tilde J[f_{\eps_n}]. 
 $$ 
 \end{Definition}
\noindent By convexity, we know that the limit $f$ satisfies 
\[
\sup_{[0,\ T]}\int_{\R^6} (|\x|+q^0+\log f)\,f\,d\q\,d\x<\infty,
\]
for any $T<\infty$. Hence, this lets us define $n_f,\; u_f $ from $f$ through Definitions \eqref{set1} and \eqref{set2}, and $\beta_f$ through Eq. \eqref{betadef}.

This suffices to pass to the limit in all the linear terms of the approximating scheme. Hence to write down the equation satisfied by the limit distribution it only remains to pass to the limit in the relaxation operator.

 The first step is to remove the truncations at both the thermodynamic fields and the support in $\tilde J[f_\eps]$.

\begin{Lemma}\label{Jtilde-J}
Let $\epsilon < 1/\bar \beta$, where $\bar \beta$ is defined in Proposition \ref{entropy_estimate}. Then the following estimate 
\[\begin{split}
&    \int_{\R^6} \left|\tilde J[f_\eps]-J[f_\eps]\right| \frac{d\q}{q^0} \, d\x\\
&\qquad\leq C\,\left(\frac{\Lambda(2/\epsilon^2)}{\epsilon^4}+\epsilon+\frac{\epsilon^2}{2}+\frac{1}{|\log (2 \epsilon)|}\right)\,\int_{\R^6} (q^0+|\x|+\log f)\,f\,d\q\,d\x,
\end{split}
\]
holds, where the constant $C$ is independent of $\eps$,
 and hence the difference converges to $0$
as $\eps \to 0$.
\end{Lemma}
\begin{proof}
To make the proof easier to follow we resort back to the notations $R, L$ and $\beta_{sup}$ for the cutoff parameters. We will show that it is possible to obtain an estimate with the claimed structure and a prefactor of
$$
\left(\Lambda(2R)\,L^2\,\beta_{sup}^2+\frac{1}{L}+\frac{1}{2R}+\frac{1}{|\log (2/\beta_{sup})|}\right)\:.
$$
This is then combined with \eqref{eq:fix} and \eqref{lambda} to conclude the proof.

     The first and main difficulty is to obtain an adequate control on 
    \[
\int_{\beta_{f_\eps}<1/\beta_{sup}\ \mbox{or}\ \beta_f>\beta_{sup}} n_{f_\eps}\,d\x.
\]
The part with $\beta_{f_\eps}<1/\beta_{sup}$ is straightforward since for example by Lemma \ref{Jmoments}
\[
\int_{\R^3} \frac{n_{f_\eps}}{\beta_{f_\eps}}\,d\x\leq \int_{\R^3} e_{f_\eps}\,d\x\leq\int_{\R^6} q^0\,f_\eps\,d\x\,d\q.
\]
Hence
\begin{equation}
\int_{\beta_{f_\eps}<1/\beta_{sup}} n_{f_\eps}\,d\x\leq \frac{1}{\beta_{sup}}\,\int_{\R^6} q^0\,{f_\eps}\,d\x\,d\q.\label{beta<1/betasup}
  \end{equation}
For $\beta_{f_\eps}>\beta_{sup}$,  we can use  Lemma \ref{lm:BCNS}, to deduce
\[\begin{split}
n_{f_\eps}-\int_{\R^3} {f_\eps}\,\frac{d\q}{q^0}&\geq \int_{\R^3} {f_\eps}\,d\q-\int_{\R^3} {f_\eps}\,\frac{d\q}{q^0}\geq C^{-1}\,\delta\,\int_{|\q|\geq \delta} {f_\eps}\,d\q,  \\
\end{split}
\]
for any $\delta>0$. If $\beta_{f_\eps}>\beta_{sup}$ then owing to \eqref{eq:sandwichK} we have that $\frac{K_1(\beta_{f_\eps})}{K_2(\beta_{f_\eps})}\geq 1-2/\beta_{sup}$. Let us abridge $\eta_{\beta_{sup}}:=2/\beta_{sup}$ in the sequel.
From \eqref{betadef} in that case
\[
n_{f_\eps}-\int_{\R^3} {f_\eps}\,\frac{d\q}{q^0}\leq \eta_{\beta_{sup}}\,n_{f_\eps}.
\]
Therefore, by taking $\delta=\eta_{\beta_{sup}}/C$, if $\beta_{f_\eps} \geq \beta_{sup}$, we find
\[
\int_{|\q|> \eta_{\beta_{sup}}/C} {f_\eps} \,d\q\leq \frac{1}{2}\,n_{f_\eps},\quad\mbox{that is}\ \int_{|\q|\leq \eta_{\beta_{sup}}/C} {f_\eps}\,d\q\geq \frac{1}{2}\,n_{f_\eps},
\]
and thus
\[\begin{split}
\int_{\beta_{f_\eps}>\beta_{sup}} n_{f_\eps}\,d\x &\leq 2\,\int_{|\q|\leq \eta_{\beta_{sup}}} {f_\eps}\,d\q\,d\x\\
& \leq 2\,\int_{|\q|\leq \eta_{\beta_{sup}}, \ {f_\eps}\leq F, \ |\x|\leq X} {f_\eps}\,d\q\,d\x+2\,\int_{{f_\eps}>F} {f_\eps}\,d\q\,d\x+2\,\int_{|x|\geq X} {f_\eps}\,d\q\,d\x\\
&\leq 2\,F\,X\,\eta_{\beta_{sup}}+2\,\left(\frac{1}{\log F}+\frac{1}{X}\right)\, \int_{\R^6} (|\x|+\log {f_\eps})\,{f_\eps}\,d\q\,d\x\\
&\leq \frac{C}{|\log \eta_{\beta_{sup}}|}\,\int_{\R^6} (|\x|+\log {f_\eps})\,{f_\eps}\,d\q\,d\x,
\end{split}
\]
by optimizing in $F$ and $X$.

Combining this part with \eqref{beta<1/betasup}, we deduce that
\begin{equation}
\int_{\beta_{f_\eps}<1/\beta_{sup}\ \mbox{or}\ \beta_{f_\eps}>\beta_{sup}} n_{f_\eps}\,d\x\leq \frac{C}{|\log \eta_{\beta_{sup}}|}\,\int_{\R^6} (q^0+|\x|+\log {f_\eps})\,{f_\eps}\,d\q\,d\x.\label{betaflowlarge}
\end{equation}

We may now turn to the estimate in the lemma.
First of all, we justify the replacement of $\tilde M$ by $M$
\[\begin{split}
&\int_{\R^6} \varphi(\q)\,\left|\frac{n_{f_\eps}}{{M}(\tilde{\beta}_{f_\eps})}\,e^{-{\tilde\beta}_{f_\eps} (\tilde u_{f_\eps})_\mu q^\mu}-\frac{n_{f_\eps}}{\tilde{M}(\tilde{\beta}_{f_\eps})} e^{-\tilde{\beta}_{f_\eps} \tilde (\tilde u_{f_\eps})_\mu q^\mu}\right|\, \frac{d\q}{q^0}d\x\\
&\quad\leq \int_{\R^6} n_{f_\eps} \frac{|\tilde M(\tilde\beta_{f_\eps})-M(\tilde\beta_{f_\eps})|}{\tilde M(\tilde\beta_{f_\eps})\,M(\tilde\beta_{f_\eps})}\,e^{-|\q|/(3\,\beta_{sup}\,L)}\,\frac{d\q}{q^0}\,d\x,
\end{split}
\]
as $\tilde\beta_{f_\eps}\geq 1/\beta_{sup}$ and since $|\tilde\u_{f_\eps}|\leq L$, $(\tilde u_{f_\eps})_\mu q^\mu \geq |\q|/(3\,L)$.
 From Lemma \ref{lcomp}, we have that 
\[
|\tilde M-M|\leq \Lambda(2R),
\]
and hence
\begin{equation}
\begin{split}
&  \int_{\R^6} \varphi(\q)\,\left|\frac{n_{f_\eps}}{{M}(\tilde{\beta}_{f_\eps})}\,e^{-{\tilde\beta}_{f_\eps} (\tilde u_{f_\eps})_\mu q^\mu}-\frac{n_{f_\eps}}{\tilde{M}(\tilde{\beta}_{f_\eps})} e^{-\tilde{\beta}_{f_\eps} (\tilde u_{f_\eps})_\mu q^\mu}\right|\, \frac{d\q}{q^0}d\x\\
&\qquad  \leq C\,\Lambda(2R)\,L^2\,\beta_{sup}^2\,\int_{\R^6}{f_\eps}\,d\q\,d\x.
\end{split}
\label{tildeMtoM}
  \end{equation}
The next step is to replace $\tilde u_{f_\eps}$ by $u_{f_\eps}$. Note that $\tilde u_{f_\eps}=u_{f_\eps}$ whenever $|\u_{f_\eps}|\leq L$. Thus
\[\begin{split}
&\int_{\R^6} \varphi(\q)\,\left|\frac{n_{f_\eps}}{{M}(\tilde{\beta}_{f_\eps})}\,e^{-{\tilde\beta}_{f_\eps} (u_{f_\eps})_\mu q^\mu}-\frac{n_{f_\eps}}{{M}(\tilde{\beta}_{f_\eps})} e^{-\tilde{\beta}_{f_\eps} \tilde (\tilde u_{f_\eps})_\mu q^\mu}\right|\, \frac{d\q}{q^0}\,d\x\\
&\quad\leq\int_{\{\x \, / \,  |\u_{f_\eps}|>L\}} \frac{n_{f_\eps}}{{M}(\tilde{\beta}_{f_\eps})}
\int_{\R^3} \left(e^{-\tilde{\beta}_{f_\eps} (\tilde u_{f_\eps})_\mu q^\mu}+e^{-\tilde{\beta}_{f_\eps} (u_{f_\eps})_\mu q^\mu}\right)\,\frac{d\q}{q^0}\,d\x\\
&\quad\leq\int_{\{\x \, / \,  |\u_{f_\eps}|>L\}} n_{f_\eps}\,d\x\leq \frac{1}{L}\int_{\R^3} n_{f_\eps}\,u_{f_\eps}\,d\x.
\end{split}
\]
As a consequence
\begin{equation}
\begin{split}
\int_{\R^6} \varphi(\q)\,\left|\frac{n_{f_\eps}}{{M}(\tilde{\beta}_{f_\eps})}\,e^{-{\tilde\beta}_{f_\eps} (u_{f_\eps})_\mu q^\mu}-\frac{n_{f_\eps}}{{M}(\tilde{\beta}_{f_\eps})} e^{-\tilde{\beta}_{f_\eps} (\tilde u_{f_\eps})_\mu q^\mu}\right|\, \frac{d\q}{q^0}\,d\x \leq \frac{1}{L}\int_{\R^6} {f_\eps} \,d\q\,d\x.
  \end{split}
  \label{tildeutou}
  \end{equation}
With a similar calculation, one may remove $\varphi(\q)$ with
\begin{equation}
\begin{split}
  &\int_{\R^6} (1-\varphi(\q))\,\frac{n_{f_\eps}}{{M}(\tilde{\beta}_{f_\eps})}\,e^{-{\tilde\beta}_{f_\eps} (u_{f_\eps})_\mu q^\mu}\,\frac{d\q}{q^0}\,d\x\leq \frac{1}{2R}\,\int_{\R^6} {f_\eps}\,d\q\,d\x.\\
\end{split}\label{removephi}
\end{equation}
Finally we use \eqref{betaflowlarge} to compute
\[\begin{split}
&\int_{\R^6} \left|\frac{n_{f_\eps}}{{M}(\tilde{\beta}_{f_\eps})}\,e^{-{\beta}_{f_\eps} (u_{f_\eps})_\mu q^\mu}-\frac{n_{f_\eps}}{{M}({\beta}_{f_\eps})} e^{-\tilde{\beta}_{f_\eps} (u_{f_\eps})_\mu q^\mu}\right|\, \frac{d\q}{q^0}d\x\\
&\quad \le \int_{\{\x/\beta_{f_\eps} >\beta_{sup}\ \mbox{or}\ \beta_{f_\eps}<\beta_{sup}\}\times \R_{\q}^3} \frac{n_{f_\eps}}{{M}({\beta}_{sup})} e^{-{\beta}_{sup} (u_{f_\eps})_\mu q^\mu} \frac{d\q}{q^0}d\x \\
&\qquad +\int_{\{\x/\beta_{f_\eps} >\beta_{sup}\ \mbox{or}\ \beta_{f_\eps}<\beta_{sup}\}\times \R_{\q}^3} \frac{n_{f_\eps}}{{M}({\beta}_{f_\eps})} e^{-{\beta}_{f_\eps} (u_{f_\eps})_\mu q^\mu} \frac{d\q}{q^0}d\x .
\end{split}
\]
By the definition of $M$, we have that for any $\beta$
\[
\int_{\R^3} \frac{1}{{M}({\beta})} e^{-{\beta} (u_\eps)_\mu q^\mu} \frac{d\q}{q^0}\leq 1.
\]
Therefore by \eqref{betaflowlarge}
\begin{equation}
  \begin{split}
    &\int_{\R^6} \left|\frac{n_{f_\eps}}{{M}(\tilde{\beta}_{f_\eps})}\,e^{-{\tilde\beta}_{f_\eps} (u_{f_\eps})_\mu q^\mu}-\frac{n_{f_\eps}}{{M}({\beta}_{f_\eps})} e^{-{\beta}_{f_\eps} (u_{f_\eps})_\mu q^\mu}\right|\, \frac{d\q}{q^0}d\x\\
    &\qquad\leq \frac{C}{|\log \eta_{\beta_{sup}}|}\,\int_{\R^6} (q^0+|\x|+\log {f_\eps})\,{f_\eps}\,d\q\,d\x.
  \end{split}
  \label{tildebetatobeta}
\end{equation}
Combining \eqref{tildebetatobeta} with \eqref{tildeMtoM}, \eqref{tildeutou}, \eqref{removephi}, concludes the proof.
\end{proof}
In our next step, we apply Theorem \ref{t:1}. By the definition of the limit, $f_\eps\,\psi(\q)$ converges weakly to $f\,\psi(\q)$. By the uniqueness of limits in distributions,
\[
\int_{\R^3} f_\eps\,\psi(\q)\,d\q\longrightarrow \int_{\R^3} f\,\psi(\q)\,d\q,
\]
strongly in $L^1([0,\ T]\times\R^3)$ for any $T>0$ and any $\psi(\q)$ with $\psi(\q)/|\q|\rightarrow 0$ as $\q\rightarrow\infty$.

In particular this implies the following strong convergence
\begin{Lemma}
\label{conv_fields}
The following assertions hold true as $\eps \to 0$:
\begin{enumerate}
\item $n_{f_\eps} $ converges to $n_f$ strongly in $L^1([0,\ T]\times\R_{\x}^3)$. 
\item $n_{f_\eps} u_{f_\eps}^\mu$ converges to $n_f u_f^\mu$ strongly in $L^1([0,\ T]\times\R_{\x}^3)$.
\item Let $\mathcal{NV}:=\{t,\,\x\in \R_+\times\R_{\x}^3 / n_f(t,\x)>0\}$. Then $u_{f_\eps} \to u_f$ a.e. in $\mathcal{NV}$. 
\item $\beta_\eps \to \beta_f$ a.e. in $\mathcal{NV}$.
\end{enumerate}
\end{Lemma}
\begin{proof}
For the first point, taking $\psi(\q)=\q/q^0$,
$$
\int_{\R_{\q}^3} q^\mu f_\eps \, \frac{d\q}{q^0} \rightarrow \int_{\R_{\q}^3} q^\mu f \, \frac{d\q}{q^0} \quad \mbox{strongly in}\ L^1([0,\ T]\times\R_{\x}^3),
$$
for any $T>0$.

Then $n_{f_\eps} \to n_f$ in $L^2(\R_{\x}^3)$ and the claim follows. We show in the same way that $n_{f_\eps} u_{f_\eps}^\mu \to n_f u_f^\mu$ in $L^1(\R_{\x}^3)$. Our third statement follows as usual from the second, possibly extracting further subsequences to have the $a.e.$ convergence of $n_{f_\eps}\,u_{f_\eps}^\mu$. For the final statement, we show as before that
$$
\frac{1}{n_{f_\eps}} \int_{\R_{\q}^3} f_\eps \frac{d\q}{q^0} \rightarrow \frac{1}{n_f}\int_{\R_{\q}^3} f \frac{d\q}{q^0} \quad \mbox{a.e. in}\ \mathcal{NV}.
$$
Hence $K_1(\beta_{f_\eps})/K_2(\beta_{f_\eps}) \to K_1(\beta_f)/K_2(\beta_f)$ a.e. $\mathcal{NV}$ which implies the result.
\end{proof}
We can now derive the limit of $J[f_\eps]$ as per
\begin{Corollary}
We have  that $J[f_\eps]\to J[f]$, a.e. in $\R_{\x}^3\times\R_{\q}^3$, as $\eps \to 0$.\label{limitJeps}
\end{Corollary}
\begin{proof}
Thanks to Lemma \ref{conv_fields}, $J[f_\eps]\to J[f]$ on $\mathcal{NV}\times \R_{\q}^3$ already. Note now that by Lemma~\ref{set2} 
$$
\int_{\R_{\q}^3} J[f_\eps] \, d\q= n_{f_\eps} u_{f_\eps}^{0}.
$$
Passing to the limit at both sides of the previous relation we get that 
$$
\int_{\R_{\q}^3} \mathcal{J} \, d\q= n_f u_f^{0}.
$$
This implies that $J[f_\eps] \to 0$ a.e. in $\mathcal{NV}^c\times \R_{\q}^3$.
\end{proof}
By combining Corollary \ref{limitJeps} with Lemma \ref{Jtilde-J}, we finally deduce that $\tilde J[f_\eps]$ converges to $J[f]$ $a.e.$ thus concluding our proof.

%%%%%%%%%%%%%%%%%%
%%%%%%%%%%%%%%%%%%%%%%%%
%%%%%%%%%%%%%%%%%%%%%%%%%%%%%%%%%
%

 \end{document}